\RequirePackage{fix-cm}
\documentclass[smallcondensed]{svjour3}     
\smartqed  
\usepackage{graphicx}
\usepackage{float}

\usepackage{amsmath,amssymb}    
\usepackage{xcolor}
\newcommand{\norm}[1]{\left\lVert#1\right\rVert}
\newcommand{\alert}[1]{\textit{\textcolor{blue}{***#1}}}
\usepackage{amsthm}
\usepackage{amsfonts}
\usepackage{float}
\usepackage{bm}
\usepackage[colorlinks=true,linkcolor=blue,citecolor=blue]{hyperref}
\usepackage{caption}
\usepackage{subcaption}
\begin{document}

\title{Optimal pricing for electricity retailers based on data-driven consumers' price-response
}

\titlerunning{Optimal pricing for electricity retailers}        

\author{Román Pérez-Santalla \and Miguel Carrión \and Carlos Ruiz
}


\institute{
R. Pérez-Santalla \at University Carlos III de Madrid, Avda. de la Universidad, 30,	28911-Legan\'es, Spain\\
\email{100420252@alumnos.uc3m.es}
\and
M. Carrión \at Escuela de Ingenier\'ia Industrial y Aeroespacial de Toledo, University of Castilla-La Mancha, Avda. Carlos III, 45071, Toledo, Spain. \\
\email{miguel.carrion@uclm.es}
\and
C. Ruiz \at
Department of Statistics \& UC3M-BS Institute for Financial Big Data (IFiBiD), University Carlos III de Madrid, Avda. de la Universidad, 30,	28911-Legan\'es, Spain\\
              \email{carlos.ruiz@uc3m.es}           
}


\maketitle

\begin{abstract}
In the present work we tackle the problem of finding the optimal price tariff to be set by a risk-averse electric retailer participating in the pool and whose customers are price-sensitive. We assume that the retailer has access to a sufficiently large smart-meter dataset from which it can statistically characterize the relationship between the tariff price and the demand load of its clients. Three different models are analyzed to predict the aggregated load as a function of the electricity prices and other parameters, as humidity or temperature. More specifically, we train linear regression (predictive) models to forecast the resulting demand load as a function of the retail price. Then we will insert this model in a quadratic optimization problem which evaluates the optimal price to be offered. This optimization problem accounts for different sources of uncertainty including consumer's response, pool prices and renewable source availability, and relies on a stochastic and risk-averse formulation. In particular, one important contribution of this work is to base the scenario generation and reduction procedure on the statistical properties of the resulting predictive model. This allows us to properly quantify (data-driven) not only the expected value but the level of uncertainty associated with the main problem parameters.
Moreover, we consider both standard forward based contracts and the recently introduced power purchase agreement contracts as risk-hedging tools for the retailer. The results are promising as profits are found for the retailer with highly competitive prices and some possible improvements are shown if richer datasets could be available in the future. A realistic case study and multiple sensitivity analyses have been performed to characterize the risk-aversion behavior of the retailer considering price-sensitive consumers. It has been assumed that the energy procurement of the retailer can be satisfied from the pool and different types of contracts. The obtained results reveal that the risk-aversion degree of the retailer strongly influences contracting decisions, whereas the price sensitiveness of consumers has a higher impact on the selling price offered.  
\end{abstract}

\section{Introduction, state of the art and contributions}
\label{chap:intro}

Intermittent and renewable units are the generation technologies that are being most installed nowadays in  power systems \cite{IRENA_2021}. The capital costs of these units are moderately high, but they have the advantage of not burning fuel in the electricity production process, which results in low operation costs. As a consequence of this, electricity prices in the spot markets can be very low if the total power output of the renewable units is high. On the contrary, electricity prices can be elevated if the renewable production is low. These high prices can be even enlarged if extra charges are required to be paid by fossil-fuel based generators to penalize the emission of greenhouse gases. For instance, this is the case of generating power plants in Europe that are forced to participate in  the EU Emissions Trading System \cite{ETS_2021}, which prices have doubled during year 2021 reaching the 64.37 euro/ton on September 27th. As a result, the volatility of electricity prices has increased significantly during the last years in those power systems with high intermittent production, \cite{Dong_2019}. 

The effect of the volatility of electricity prices may be particularly negative for the economic objectives of electricity retailers. In a power system framework, retailers are entities that purchase electricity in different markets and through bilateral contracts to sell it afterwards to the end users of electricity. The financial revenue obtained by retailers comes from the sale of electricity to their clients, whereas the costs incurred by retailers are those derived from its procurement of electricity. Then, if the purchase price in the electricity markets is higher than expected, the acquisition cost of electricity for retailers will be also high and they may even suffer from losses. In the same manner, this situation may also happen if the selling price offered by retailers to their clients is too low. In this case, the revenue obtained from selling energy could be smaller than the procurement cost of the retailer. 
Moreover, with the progressive integration of smart grid technologies within distribution systems, we may expect a increasing price-demand elasticity of consumers, which needs to be accurately characterized.
As a consequence of this, the appropriate determination of the electricity procurement strategy and the selling price offered to their clients by retailers is key for achieving  profitability.

In this work we investigate the decisions (energy procurement and selling price determination strategies) to be taken by a retailer to maximize its expected profit within a time of use (ToU) tariff. A ToU rate is that one which the selling price offered to clients vary for each trading period of the day.  The idea behind the ToU tariff is to communicate to clients the price offered so that they can choose to use more or less electricity within the time period that the offer is valid. We set this investigation in the British electricity market \cite{GBmarket}. For that we use real world smart-meters consumption dataset which spans almost the whole year of 2013 \cite{SmartMeter}. The dataset is comprised of half-hourly electric demand and price data for 1000 British households, along with some other related variables such as apparent temperature, humidity or precipitation. 
Once the response of consumers to electricity prices is characterized and modeled through a predictive model, a  risk-averse two-stage stochastic programming model is formulated to determine the energy procurement and selling price by retailers. Different energy procurement sources are considered in this problem.

As more countries are installing smart meters to record household consumption \cite{AKHAVANHEJAZI201891} retailers have now more information than ever to cater prices to consumers. Some previous works have approached this price design by clustering customers with similar behaviours to offer a similar tariff \cite{articleyang}. Another approach taken by \cite{6266720} is to formulate a game between the retailer and different types of consumers (such as industrial and households). Reference \cite{Feng_2020} models the retailer-consumer interaction as Stackelberg competition (retailer as leader and customers as followers) to produce a bilevel optimization problem. 

Many works have been also developed to formulate the different problems faced by electricity retailers under uncertainty. Some of the most relevant models derived in the last years are described below.
In this sense, reference \cite{Carrion_2007} formulated for the first time a two-stage stochastic programming problem to decide the purchasing strategy and the selling price offered by an electricity retailer in a medium-term planning horizon. This work was expanded in \cite{Carrion_2008}, where a bilevel problem was formulated to take into account explicitly the decision-making problem faced by electricity end users that desire to select optimally their electricity suppliers.
Additionally, reference \cite{Hatami_2011} incorporated time-of-use rates in the problem formulated in  \cite{Carrion_2007} modelling also the elasticity of the demand. 
The authors of \cite{Kettunen_2010} developed a novel multi-stage stochastic optimization approach to decide the contract portfolio of retailers. The numerical results obtained using this procedure indicate that the correlation between price and load has a direct influence on the decisions of the retailer.  %
In Reference \cite{Garcia-Bertrand_2013} it is proposed a novel procedure that considers selling price modifications to incentive load shifts intended to increase the expected profit of the retailer. It is observed that this new price scheme is able to increase the profit of retailers while the payments of the consumers are reduced.
The authors of \cite{Nojavan_2017} formulate the bilateral contracting and selling price determination of a retailer in a smart grid framework considering distributed generation, energy storage systems and a demand response program. In the same manner that \cite{Garcia-Bertrand_2013}, this work concludes that the demand response program is able to increase the expected profit of the retailer and decrease the selling price offered to the clients.
In Reference \cite{Deng_2020} a procedure for designing real-time pricing rates is proposed for different types of consumers. The novelty of this procedure is the usage of the downside risk constraints method to decide the risk strategy of the retailer. The numerical results indicate that the risk faced by the retailer may be efficiently controlled by using this method.
Finally, the recent work \cite{Feng_2020} derives a data-driven approach to design time-of-use tariffs by modeling a Stackelberg game between the retailer and strategic consumers. The obtained tariffs are able to shave demand peaks and fill demand valleys.

What we propose in this work is to first identify the sensitivity of customers by using historical data to model customer's responses to prices (exhibited by the demand). Customer elasticity has become increasingly important as it can help in smoothing the peaks of consumption which have become ever increasing as electric demand grows and renewable sources play a bigger role in the system. Furthermore, new technologies on the electric grid have allowed to make this real time communication viable \cite{Electricgrid}.
 To measure the sensitivity of customers to the price offered we will take the coefficient of a linear regression model corresponding to the price variable. Several models of different size and with different variables are analyzed to figure out which variables better improve the prediction accuracy. Among the most meaningful variables were lagged values of the electric demand and temperature data \cite{Hyndmanonline}. After selecting 3 main models of different size (one with 8 predictors, another with 30 and finally a complex model which used a different linear model for each time period to be predicted \cite{Hyndman12}) comparisons are made to find the most suitable for the present work. These comparisons focused not only in the accuracy of the model but also on the impact within the model of the price predictor. Moreover, we seek to characterize statistically  the distributions of the estimated model parameters. This information will be latter used to generate a representative set scenarios for the stochastic programming problem. \par

The next step is to formulate the profit maximization problem faced by retailers in the British electricity market. Like most European markets, the British electricity market is mainly based on the interactions through different trading floors between producers (energy generating plants of various types) who sell the energy to large consumers or to retailers,  who sell the energy to the final consumers (individuals or companies who purchase electric energy for consumption). In this work we have established 3 market mechanisms for the retailer to purchase energy. First, the retailer may purchase energy from the pool, which is a market place  where producers, retailers and consumers trade electricity with a typical planning horizon of one day divided in hourly or half-hourly periods. We consider the pool to be an unlimited source of energy (there is no cap to the amount of energy that can be purchased from the pool) which is a fair approximation as most producers in Britain participate in this market. The second way in which the retailer may purchase energy is through a forward base contract which is a bilateral agreement between a producer and a retailer where a certain amount of energy is purchased during a certain time frame for the agreed price. In this type of contract, the producer must deliver the amount of energy contracted, so it is normally linked to non-renewable sources where the power producer can always deliver the promised energy (dispatchable units). Finally, the third option we have modeled for the retailer to purchase energy from is a power purchase agreement (PPA) which is a different type of bilateral contract where the power delivered in each trading period is not necessarily fixed, as it can depend on other factors, such as weather for renewable sources (non-dispatchable) \cite{libro}. An example of PPA may be a bilateral contract signed with a solar photovoltaic power plant, which production is different in each hour of the day. 
Considering these three trading floors, a risk-averse two-stage stochastic programming problem is formulated to decide the energy procured by the retailer in each available trading floor and to determine the selling price offered to their clients. The stochastic variables considered in this model are: 1) pool prices, 2) energy-availability of the PPA and 3) response of end users of electricity to the prices offered by the retailer. The most important decisions to obtain using this formulation are the power contracted from forward and PPA contracts and the selling prices offered by the retailer to their clients. Considering a ToU rate, the selling prices offered by the retailer can be different for each half-hour period and they are indexed to pool prices. The risk-aversion of the retailer is modeled using the conditional value-at-risk (CVaR) of the profit distribution.


The main contributions of this work are fivefold:
\begin{enumerate}
    \item[1)] to use real smart meter data to quantify the price-load response of consumers in terms of predictive linear regression model.
    \item[2)] to use the statistical properties of above model and real world datasets to generate a meaningful set of scenarios of the uncertain parameters: price-sensitivity of consumers, renewable availability and pool prices.
    \item[3)] to consider both forward based contracts (expensive but certain availability) with the recently introduced PPA contracts (cheap but uncertain availability), as risk-hedging tools for a retailer participating in a pool market.
    \item[4)] to explicitly insert the forecasting model in 1) as a constraint in the two-stage stochastic problem faced by the retailer together with the scenarios derived in 2).
    \item[5)] to study through realistic simulations the main properties and interactions between PPA and forward based contract prices, and the impact of level of consumers elasticity on the retailer's optimal tariffs and profits.
\end{enumerate}
We believe that the data-driven methodology presented in this work is general and can be applied to other price-load smart-meter datasets. We hope that these will become more and more available in the near future with the progressive integration of smart-grid technologies.

\section{Price-demand data-driven characterization}
\label{chap:model}

As mentioned before, the first step in solving the optimization problem is to characterize, as accurately as possible, the demand-response level of price-sensitive consumers, and how this is impacted by other exogenous variables.

We propose a novel data-driven approach to i) infer the relationship between the consumer's demand and the retail price (tariff) and ii) insert it explicitly in the retailer's decision making model (stochastic optimization). 

Lets assume that we have a sample (data set) of $n$ observations of the type $S_N=\{(y^1,x_1^1,\dots,y^1_k),\dots,(y^n,x_1^n,\dots,y^n_k)\}$, where $y,x_j^i\in\mathbb{R}$, $i=1,\dots,n$, $j=1,\dots,k$. Lets also assume that both the response variable $y$ (demand) and the first of the explanatory variables, $x_1$ (price), are decision variables within our stochastic model, while the remaining variables $x_{-1}=\{x_2,\dots,x_k\}$ can be considered as contextual information (covariates), known at the time of decision making. Hence, we seek to characterize the relationship between them through a prediction model of the type $y=f(x_1|x_{-1})$ that can be inserted within the retailers optimization problem. In particular, we can consider the following linear model:
$$y=\beta_0 + \beta_1 x_1 + \sum_{j=2}^k\beta_j x_j$$
where $\beta_0,\beta_1,\dots,\beta_k$ are the linear regression coefficients to be estimated from the sample $S_N$ by Ordinary Least Squares (OLS) or any other statistical approach. Once estimated, the resulting model can be expressed as 
\begin{equation}\label{eq:regre_1}
    y=\hat{\beta}_1 x_1 + \hat{D}(x_{-1})
\end{equation}
where $\hat{D}(x_{-1})=\hat{\beta}_0 + \sum_{j=2}^k \hat{\beta}_j x_j + \epsilon$ is fully characterized by the realization of the contextual information $x_{-1}$, being $\epsilon$ the residual (error). 
If $y$ and $x_1$ are considered optimization variables, the data driven function (\ref{eq:regre_1}) can be easily incorporated in the optimization problem in the form of a linear constraint, and without increasing substantially the computational complexity of the resulting optimization problem.

Moreover, considering a linear model like  (\ref{eq:regre_1}) within an stochastic optimisation framework, presents two important advantages: 
\begin{itemize}
    \item[i)] the statistical distribution of the residuals $\epsilon$ and the estimated $\hat{\beta}_j$ for $j=0,\dots,k$ is well characterized, so that realistic data-driven scenarios of this linear relationship can be generated.
    \item[ii)] $\hat{\beta}_1$ can be directly interpreted as the degree by which demand ($y$) is altered by a marginal change of the tariff ($x_1$), which is closely linked to the concept of demand elasticity with important economic implications.
\end{itemize}

In the present model, the contextual information would include meteorological and calendar information, together with past realizations of prices and demands (lags). Similarly, in the stochastic optimization contest, we will assume that the retail price $x_1$ and the demand-response by the consumers $y$ are second stage decisions and hence, conditioned by the realization of uncertainty (scenario dependent). Moreover, model (\ref{eq:regre_1}) will be extended with a temporal dependency, as we employ a half-hourly time resolution. This will be further explained in the following sections. 

\subsection{Scenario generation}\label{scenario_gen}
For the stochastic optimization problem we need the most accurate characterization of its uncertain parameters. In the current work we will consider that these include half hourly pool prices, solar availability, and the price coefficient ($\beta_1$) of the linear model introduced in the previous section. We deal with the task of creating this representation by utilizing a large number of randomly generated scenarios. Each of these is generated by concatenating random possible realizations of pool prices, solar availability and the price coefficient of the linear model. The process for generating the pool and solar scenarios is dealt with in Appendix \ref{Anexo:Scenario}. To generate scenarios for the regression coefficient of the price predictor we utilize the distribution of the parameter estimate $\hat{\beta}_1 $ instead of its expected value $\mathbb{E}[\hat{\beta}_1]$. For clarity, we make use of the matrix formulation of the linear system (\ref{eq:regre_1}) as follows:
\begin{equation}
    \mathbf{y} = \mathbf{X}\boldsymbol{\beta} + \boldsymbol{\epsilon},  \quad \text{where} \quad \mathbf{y},\boldsymbol{\epsilon} \in \mathbb{R}^{n \times 1}, \quad \mathbf{X} \in \mathbb{R}^{n \times k}, \quad \boldsymbol{\beta} \in \mathbb{R}^{k \times 1}
\end{equation}

The first step is to identify the distribution of $\hat{\beta}_1$ under OLS and the assumption that the model residuals are i.i.d. with $\mathbb{E}[\epsilon_i] = \mu_\epsilon$, $\mathbb{V}\text{ar}[\epsilon_i] = \sigma^2_\epsilon$ for $i=1,\dots,n$. Using the Hajek-Sidak CLT and assuming that its condition is satisfied (a heuristic argument for this assumption can be found in Appendix \ref{Anexo:CLT} for the data set employed in Section \ref{sec:CaseStudy}), the limit distribution for $\hat{\beta}_1$ is Normal:
\begin{equation}
    \hat{\beta}_1 \to \mathcal{N}\left(\beta_1+\mu_\epsilon \sum_{j=1}^n[(\mathbf{X}^T\mathbf{X})^{-1}\mathbf{X}^T]_{1j}, \sigma^2_\epsilon\sum_{j=1}^n[(\mathbf{X}^T\mathbf{X})^{-1}\mathbf{X}^T]^2_{1j}\right)
\end{equation}
Replacing $\beta_1$ with its estimate $\mathbb{E}[\hat{\beta}_1]$ and $\mu_\epsilon$ with the residual sample mean $\bar{\boldsymbol{\epsilon}}$ which can be assumed to be 0, the distribution can be written as:
\begin{equation}\label{price_dist}
    \hat{\beta}_1 \to \mathcal{N}\left(\mathbb{E}[\hat{\beta}_1], \sigma^2_\epsilon\sum_{j=1}^n[(\mathbf{X}^T\mathbf{X})^{-1}\mathbf{X}^T]^2_{1j}\right)
\end{equation}
Scenarios for the regression coefficient of the price predictor are then sampled from this distribution.

\section{Stochastic Optimization Problem}
\label{chap:LP}

This section describes the mathematical formulation of the retailer problem. This problem is formulated using a risk-averse two-stage stochastic approach in which the demand-response of price-sensitive consumers  is explicitly characterized. The final problem is recast as a quadratic programming problem. 

\subsection{Notation}\label{glosario}

The notation used to formulate the optimization problem is included below for quick reference.

\subsection*{Indices and sets}
\begin{itemize}
\item{$D$: Set of days, indexed by $d$}
\item{$T$: Set of half-hourly periods, indexed by $t$}
\item{{$T_d$}:  Set of time periods for day $d$}
\item{$\Omega$: Set of scenarios, indexed by $\omega$}
\end{itemize}

\subsection*{Variables}
\begin{itemize}
\item{$c_{t,\omega}$: Retailer's cost for each half-hourly period and each scenario}
\item{$d_{t,\omega}$: Total demand of consumers contracting power from the retailer per half-hourly period and scenario}
\item{{$p^{\rm B}$}: Energy contracted by the retailer from forward base contracts (fixed quantity per half-hourly period $p^{\rm B}$ with a fixed price $\bar{\lambda}^{\rm B}$)}
\item{{$p^{\rm C,PPA}$}: Energy contracted by the retailer from the PPA contract for each halg-hourly period}
\item{$p^{\rm P}_{t,\omega}$: Energy bought by the retailer from the pool per half-hourly period and scenario at price $\bar{\lambda}^{\rm P}_{t,\omega}$}
\item{{$p^{\rm PPA}_{t,\omega}$}: Energy purchased from PPA contract with a renewable power producer subject to the energy contracted $p^{\rm C,PPA}$. The energy purchased from this contract at each half-hourly period and scenario depends on the availability of the renewable source, $A^{\rm PPA}_{t,\omega}\in[0,1]$. The energy purchased through the PPA contract is priced at $\bar{\lambda}^{\rm PPA}$}
\item{$r_{t,\omega}$:  Retailer's revenue for each half-hourly period and each scenario}
\item{{$s_\omega$}:  Auxiliary variables for CVaR formulation}
\item{{$\eta$}:  Auxiliary variable for CVaR formulation equivalent to the VaR at the optimal solution of the stochastic problem.}
\item{{$\lambda^{\rm E}_{t}$}: Variable part of the price offered by the retailer}
\item{$\lambda^{\rm R}_{t,\omega}$: Price offered per half-hourly period and scenario to consumers by the retailer}
\item{$\Theta$: Set of all the above decision variables}
\end{itemize}

\subsection*{Parameters}
\begin{itemize}
\item{$A^{\rm PPA}_{t,\omega}$:  Availability of the renewable source for each half-hourly period and each scenario}
\item{{$\hat{D}_t$}:  Half-hourly demand effects outside the price-coefficient interaction} 
\item{{$\pi_\omega$}: Probability assigned to each scenario}
\item{$P^{\rm B}_{\rm max}$:  Maximum energy available for purchase from the forward base contracts in each half-hourly period}
\item{$P^{\rm PPA}_{\rm max}$:  Maximum energy available for purchase from the PPA contract in each half-hourly period}
\item{{$\alpha$}:  Fraction of the distribution to be used in the CVaR calculation (as $1-\alpha$)}
\item{$\beta_{\omega}$:  Coefficient of the price predictor for each scenario}
\item{{$\gamma$}:  Parameter for establishing bounds for every element of the variable part of the price ($\lambda^{\rm E}_{t}$)}
\item{{$\bar{\lambda}^\text{E} $}:  Mean value of the variable part of the price ($\lambda^{\rm E}_{t}$)}
\item{$\bar{\lambda}^{\rm B}$:  Price for energy contracted from forward base contracts}
\item{$\lambda^{\rm P}_{t,\omega}$:  Half-hourly pool prices for each scenario}
\item{$\bar{\lambda}^{\rm PPA}$:  Price for energy
contracted from the PPA contract}
\item{{$\chi$}:  Parameter weighing the importance of the CVaR against the profit}
\end{itemize}


\subsection{Formulation}
The formulation for the decision-making problem faced by a retailer that desires to determine its energy procurement strategy and the selling prices offered to clients is included below. This formulation considers explicitly a data-driven model to characterize and anticipate the price-sensitiveness of the retailer's clients. The electricity procurement options of the retailer are purchasing from i) the pool at variable and uncertain price, ii) a forward contract with fixed and identical prices and quantities for every half-hourly period, and iii) a PPA with a renewable energy provider that offers a variable and uncertain production at fixed price. 

As stated above, the formulation of this problem corresponds to a risk-averse two-stage stochastic programming approach. The uncertain parameters considered in this problem are the pool prices, ${\lambda}^{\rm P}_{t,\omega}$, the elasticity of the demand of the clients of the retailer, $\beta_{\omega}$, and the availability of the PPA,  ${A}^{\rm PPA}_{t,\omega}$. These parameters are characterized using a set of scenarios $\omega \in \Omega$. The risk-aversion of the retailer is modeled by the CVaR of the profit. The first-stage decisions of this problem are the retailer-dependent part of the selling price, ${\lambda}^{\rm E}_t$, and the half-hourly energy contracted from forward and PPA contracts, ${p}^{\rm B}$ and ${p}^{\rm C,PPA}$, respectively.

\begin{subequations}\label{risk_CVaR}
\begin{align}
&\max_{\Theta}\quad \left(1-\chi\right)\sum_{\omega \in \Omega}\sum_{t\in T}\pi_\omega\left(r_{t,\omega}-c_{t,\omega}\right) + \chi \left(\eta - \frac{1}{1-\alpha} \sum_{\omega \in \Omega}\pi_\omega s_\omega\right) \label{PCPV_sto_FO}\\
&\mbox{\ s.t.}\notag\\
&\qquad r_{t,\omega} = \lambda^{\rm R}_{t,\omega}d_{t,\omega},\quad\forall t,\omega\label{PCPV_sto_01}\\
&\qquad  c_{t,\omega} = \bar{\lambda}^{\rm P}_{t,\omega} p^{\rm P}_{t,\omega}+\bar{\lambda}^{\rm B}p^{\rm B}+\bar{\lambda}^{\rm PPA}p^{\rm PPA}_{t,\omega},\quad\forall t,\omega\label{PCPV_sto_02}\\
&\qquad  d_{t,\omega} = \beta_\omega\lambda^{\rm R}_{t,\omega}+\hat{D}_{t},\quad\forall t,\omega\label{PCPV_Sto_03}\\
&\qquad  d_{t,\omega} = p^{\rm P}_{t,\omega}+p^{\rm B}+p^{\rm PPA}_{t,\omega},\quad\forall t,\omega\label{PCPV_sto_04}\\
& \qquad 0\leq p^{\rm B}\leq P^{\rm B}_{\rm max} \label{PCPV_sto_05} \\
&\qquad   p^{\rm PPA}_{t,\omega} = A^{\rm PPA}_{t,\omega}p^{\rm C,PPA},\quad\forall t,\omega\label{PCPV_sto_06}\\
&\qquad  0\leq p^{\rm C,PPA}\leq P^{\rm PPA}_{\rm max} \label{PCPV_sto_07}\\
&\qquad  \lambda^{\rm R}_{t,\omega} =  \lambda^{\rm E}_t+ \left(\lambda^{\rm P}_{t,\omega}-\bar{\lambda}^\text{E}\right),\quad\forall t,\omega  \label{PCPV_sto_08} \\
&\qquad   \frac{1}{|T_d|}\sum_{t\in T_d}\lambda_t^\text{E} =\bar{\lambda}^\text{E} \label{PCPV_sto_09} \\
&\qquad  (1-\gamma)\bar{\lambda}^\text{E}  \leq \lambda_t^\text{E} \leq (1+\gamma)\bar{\lambda}^\text{E} ,\quad\forall t \label{PCPV_sto_10} \\
%
%
&\qquad  \eta - \sum_{t \in T}\left(r_{t,\omega}-c_{t,\omega}\right) \leq s_\omega, \quad \forall \omega \label{PCPV_sto_12}\\
&\qquad  0\leq s_\omega, \quad \forall \omega \label{PCPV_sto_13} 
\end{align}
\end{subequations}
where
%
$\Theta=\{d_{t,\omega}, p^{\rm B}, p^{\rm C,PPA}, p^{\rm P}_{t,\omega},p^{\rm PPA}_{t,\omega},r_{t,\omega},s_\omega,\lambda^{\rm E}_t, \lambda^{\rm R}_{t,\omega},\eta\}$ is the set of optimization variables.

The objective function \eqref{PCPV_sto_FO} is the weighted sum of the expected profit of the retailer and the CVaR of the profit. The CVaR is used to model the risk faced by the retailer and it is equal to the average value of the $(1-\alpha)$ scenarios with lowest profit. Parameter $\chi \in [0,1]$ is used to model the risk aversion of the retailer. If  $\chi=0$, the CVaR term is neglected and the retailer behaves as a risk-neutral decision maker. Values of $\chi$ greater than 0 represent different risk-aversion degrees. The expected profit is computed as the sum of the profits over the set of scenarios multiplied by their respective probabilities, $\pi_\omega$. Note that we make use of the linear CVaR formulation proposed by  \cite{Rockafellar} and \cite{Rockafellar00optimizationof}.

Equation \eqref{PCPV_sto_01} computes the revenue $r_{t,\omega}$ obtained by the retailer for each time period $t$ and scenario $\omega$ as the demand $d_{t,\omega}$ multiplied by the price of the corresponding time period $\lambda^{\rm R}_{t,\omega}$.
The cost incurred by the retailer in each period $t$ and scenario $\omega$, $c_{t,\omega}$, is formulated in equation \eqref{PCPV_sto_02} and it is equal the sum of the cost of purchasing energy in the pool, $\bar{\lambda}^{\rm P}_{t,\omega}p^{\rm P}_{t,\omega}$, the cost of purchasing through the bilateral contract, $\bar{\lambda}^{\rm B}p^{\rm B}$, and the cost of purchasing from the PPA, $\bar{\lambda}^{\rm PPA}p^{\rm PPA}_{t,\omega}$.

Equation \eqref{PCPV_Sto_03} is a scenario-dependent extension of the linear model introduced in (\ref{eq:regre_1}), and represents that the total demand of the consumers is given by the linear relationship between the price offered by the retailer, $\lambda^{\rm R}_{t,\omega}$, and the consumer response quantified using the coefficient of the price predictor, $\beta_\omega$, plus the demand that is independent of the selling price, $\hat{D}_{t}$, which is characterized by contextual information. 

The energy balance of the retailer in each period and scenario is formulated by equation \eqref{PCPV_Sto_03}, which enforces that the total demand of the clients of the retailer in each period and scenario has to be equal to the energy purchased in the pool plus the energy contracted through forward and PPA contracts.  
Constraints \eqref{PCPV_sto_05} limit the energy contracted from the forward contract.
Constraints \eqref{PCPV_sto_06} and \eqref{PCPV_sto_07} formulate the energy contracted through the PPA. Note that the energy associated with period $t$ and scenario $\omega$, $p^{\rm PPA}_{t,\omega}$, depends on the energy contracted $p^{\rm PPA}$ and on the availability of the production $A^{\rm PPA}_{t,\omega}$. Note that the availability $A^{\rm PPA}_{t,\omega}$ ranges in the interval $[0,1]$. This parameter is used to model availability of the renewable source for a given period and scenario. Therefore, if $A^{\rm PPA}_{t,\omega}=1$ the whole capacity contracted will be available for the retailer. On the contrary, if $A^{\rm PPA}_{t,\omega}=0$ it is not possible to obtain energy from the PPA in period $t$ and scenario $\omega$.

Constraints \eqref{PCPV_sto_08}-\eqref{PCPV_sto_10} formulate mathematically the ToU rate offered by the retailer to its clients. As it is usual in many European countries, this paper assumes that the selling prices associated with the ToU rate are indexed to pool prices. Then, as stated by constraints \eqref{PCPV_sto_08}, the price that the clients of the retailer have to pay in each half-hourly period, $\lambda^{\rm R}_{t,\omega}$ is equal to a price term decided by the retailer, $\lambda^{\rm E}_{t}$, plus an additional term dependent on the pool price, $\left(\lambda^{\rm P}_{t,\omega}-\bar{\lambda}^{\rm E}\right)$. Symbol $\bar{\lambda}^{\rm E}$ refers to the average value of the price term decided by the retailer, which is a known parameter of the problem \eqref{PCPV_sto_09}. Additionally, parameter $\bar{\lambda}^{\rm E}$ is also used to bound the value of the price term  $\lambda^{\rm E}_{t}$ in each half-hourly period by constraints \eqref{PCPV_sto_10}. Parameter $\gamma$ in \eqref{PCPV_sto_10} is used to establish these limits. Because of pool prices are characterized as a stochastic parameter modeled using a set of scenarios, note that the selling price  $\lambda^{\rm R}_{t,\omega}$ is also dependent of the scenario index $\omega$.

Finally, constraints \eqref{PCPV_sto_12} and \eqref{PCPV_sto_13} are used to model the CVaR (\cite{Rockafellar} and \cite{Rockafellar00optimizationof}). The value of $\eta$ equals the Value at Risk (VaR) at the optimal solution of problem \eqref{risk_CVaR}.

Observe that problem \eqref{PCPV_sto_FO}-\eqref{PCPV_sto_13} is a quadratic linear programming problem that can be solved using commercial solvers with global optimality guarantees. The quadratic term results from replacing the demand $d_{t,\omega}$ expressed in \eqref{PCPV_Sto_03} into constraint \eqref{PCPV_sto_01}, which results in $r_{t,\omega} = \beta_\omega(\lambda^{\rm R}_{t,\omega})^2+\hat{D}_{t}\lambda^{\rm R}_{t,\omega}$. 

\section{Case Study}\label{sec:CaseStudy}

\subsection{The dataset}\label{dataset}
The dataset utilized was collected by the UK Power Networks for investigating the differences between grid users with a dynamic tariff and those with a standard tariff. The half-hourly data \cite{SmartMeter}, taken from homes in London in 2013, has around 1100 customers using a dynamic tariff (referred to as ToU, standing for Time of Use, from now on) and the remainder of the dataset, around 4500 customers, using a standard tariff. We are currently interested only in those with the ToU tariff as for those we can include the price predictor. The information about the electric price for a certain time period was given to customers through the use of smart meters which would offer 3 different prices depending on external factors. The prices issued are High (67.20 p/kWh), Low (3.99 p/kWh) and normal (11.76 p/kWh), where p stands for pence or 1/100 of a British pound (£). The dataset includes half-hourly information about electric consumption for the ToU customers and multiple other variables to be used as predictors, including price. This is not the most adequate setting to infer the price-response of customers i.e., it will be convenient to observe a higher number of different price levels or even a continuous price signal. However, we show through Section \ref{ch:pred} that the resulting inferred price-response is low but still statistically meaningful for some models. In other words, we are able to characterize a significant change in the expected load when customers are offered a different price level for the next hour (e.g., from the Low to the High retail price). This flexibility may arise from several factors: the possibility of displacing the use of some appliances to cheaper hours (e.g., washing machine), to avoid electricity waste in expensive hours (e.g., turn off the lights that are not needed), the potential use of electrical storage, etc.

Figure \ref{Comparacion_Consumo} depicts two households' energy consumption during a week. Comparing these two profiles with the mean consumption of the 1100 consumers (also represented in the same figure) it is clear that individual household consumption signals present more noise and variability than said mean. This makes characterizing and predicting the mean consumption pattern much more feasible than those of individual households.
\begin{figure}[H]
    \centering
    \includegraphics[width=\textwidth]{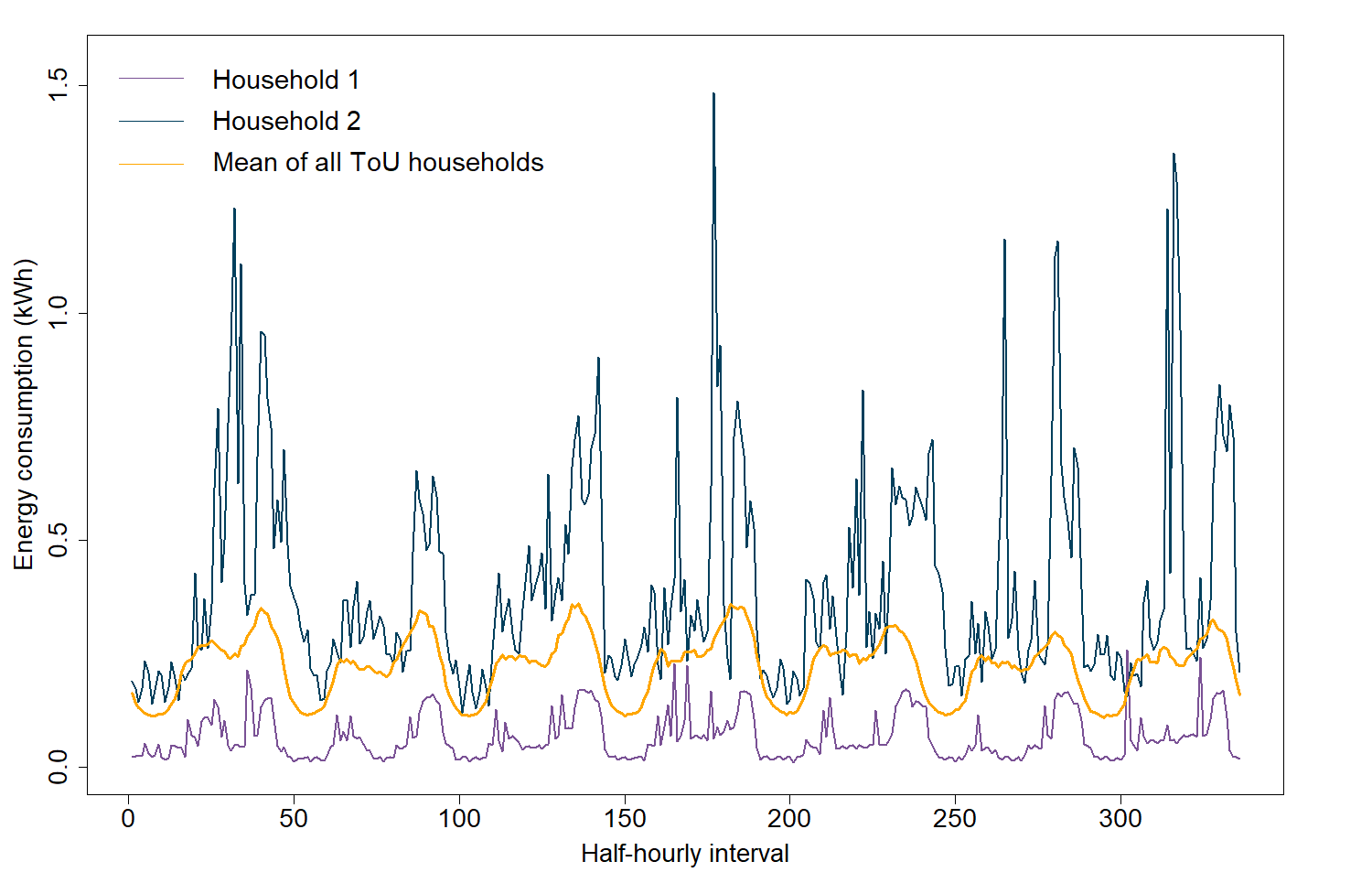}
    \caption{Plot of energy consumption of 2 random households and mean of consumption of households with ToU tariff for the week starting the first of April of 2013.}
    \label{Comparacion_Consumo}
\end{figure}

This is one of the most complete openly available datasets where variable real-time electricity prices are offered to disaggregated consumers. However there are a few shortcomings which must be addressed. The dataset spans only 1 year (2013) and in the present work we are interested in considering the retailer's options for a month, this month being chosen arbitrarily as December. This means that we are left with only 11 months of data to fit the model. If we focus on linear models, predicting accurately the half hourly-demand for an entire month is challenging, as variables such as lags of the electric consumption of the customers or weather data are necessary. For this reason we have reached the compromise of developing a model which can predict accurately the demand for the next 24 hours (that is, the nearest lagged variables are at least 48 half-hourly periods away from the period to be predicted). What this means is that the predictions would require information which would be unknown at the time of the decision. A remedy could be found by using variables from different years (however the dataset does not allow for this as it is limited to 1 year) or iterate through the testing window so that predictions for a given day can be used as input data for the following one. As solving the optimization problem for just one day, which would comply with the model's predictors, was deemed too unrealistic, we have chosen to operate on the assumption that this shortcoming could be fixed with a richer dataset. This assumption is well justified as retailers would rapidly accumulate real time (and price dependent) consumption data as smart meters penetrate among its consumers. Hence, we have proceeded to work with December as our testing month where the different contracts and tariffs signed in the first stage hold.

There is also a problem stemming from the price being offered in just 3 levels, with the normal price being the most prevalent by a large margin. This necessarily makes the price predictor have low impact with respect to the others if any reasonable degree of accuracy in predictions is desired. The importance of this predictor will be analysed for the models explored, but it is certain that with a dataset with more different price signals, as explained above, the impact of the predictor would be of greater significance.

For these reasons, in the following sections we want to provide a general methodology that can be employed by electricity retailers with access to historical consumption data from its price-sensitive costumers. Hence, the aim is not to obtain the best forecasting model, but rather to explore how this data-driven model can be included implicitly on the decision making process of a retailer and better characterize (statistically) potential uncertainty sources.

\subsection{Predictive models: demand vs price}\label{ch:pred}

We explore 3 different linear models to predict the aggregate consumers demand as a function of the electricity price and other covariates, for the month of December. These are dubbed Small Model, Large Model and Combined Model. A brief description of each model follows.
\subsubsection{Small Model}
This is the simplest model. The predictors used are the following:
\begin{itemize}
    \item \textbf{Apparent temperature:} The apparent temperature of the half-hourly period. It is a combination of air temperature, wind speed and humidity.
    \item \textbf{Humidity:} The atmospheric humidity of the half-hourly period.
    \item \textbf{Demand lags:} Two lags are used, the electric demand of 48 periods earlier (1 day) and the electric demand of 336 periods earlier (1 week).
    \item \textbf{Month:} The month to which the half-hourly period belongs.
    \item \textbf{Price:} The price signal communicated to the customers for the half-hourly period. Intentionally kept as a continuous variable despite its categorical nature to further utilize its corresponding coefficient.
\end{itemize}
Also all 3 interactions between \textbf{Month}, \textbf{Demand lags} and \textbf{Apparent temperature} are included. 
\subsubsection{Large model}
This is a larger model (30 predictors) expanding on the previous Small Model with the following additional predictors:
\begin{itemize}
    \item \textbf{Temperature:} The temperature of the half-hourly period.
    \item \textbf{Splines:} A fit of a 6 knot cubic spline to the data to quantify annual effects in demand. The corresponding curve to the test data is a continuation of this fitted curve. This approach stems from \cite{Hyndman10}.
    \item \textbf{Day of week:} 6 dummy variables coding the day of the week for the half-hourly period.
    \item \textbf{Demand lags:} Lagged electric demand in intervals of 48, 96, 144, 192, 240, 288 and 336 half-hourly periods to the corresponding period.
    \item \textbf{Temperature lags:} Lagged temperature in intervals of 48, 96, 144, 192, 240, 288 and 336 half-hourly periods to the corresponding period.
    \item \textbf{Minimum demand:} Minimum electric demand registered in the previous 24 hours to the half-hourly time period.
    \item \textbf{Maximum demand:} Maximum electric demand registered in the previous 24 hours to the half-hourly time period.
    \item \textbf{Mean demand:} Mean electric demand registered in the previous 24 hours to the half-hourly time period.
    \item \textbf{Minimum temperature:} Minimum temperature registered in the previous 24 hours to the half-hourly time period.
    \item \textbf{Maximum temperature:} Maximum temperature registered in the previous 24 hours to the half-hourly time period.
    \item \textbf{Mean temperature:} Mean temperature registered in the previous 24 hours to the half-hourly time period.
    \item \textbf{Holiday:} Dummy variable representing if the half-hourly period corresponds to a date categorized as a holiday.
\end{itemize}
\subsubsection{Combined Model}
This model utilizes the same predictors as the Large Model. However, this model is actually a combination of 48 different linear models each being tasked with the prediction of a single half-hourly period of the day, giving it added flexibility. This model was inspired by the work of \cite{Hyndman10} and \cite{Hyndman12}.
\subsubsection{Model Comparisons}\label{comparisons}
We start by comparing the metrics of Root Mean Squared Error (RMSE), Mean Squared Error (MAE), and $R^2$ of model performance across the three models. 
\begin{table}[H]
\begin{tabular}{|l|l|l|l|l|l|}
\hline
       & MAE Train & RMSE Train & MAE Test & RMSE Test & $R^2$ \\ \hline
Small Model & 12.38     & 18.92      & 14.54    & 20.80     & 0.945                \\ \hline
Large Model & 12.42     & 18.18      & 13.51    & 19.53     & 0.949                \\ \hline
Combined Model & 8.99     & 12.86       & 13.22    & 19.18    & 0.974                \\ \hline
\end{tabular}
\caption{\label{tabla1}Table representing the aforementioned metrics of model performance.}
\end{table}
In Table \ref{tabla1} we can see that although the Combined Model's overall performance is better, there appear to be some symptoms of over fitting, as the differences between test and train metrics are much larger than in both other models. Barring this possible over fitting problem in the Combined Model, the models' performances are reasonably good across the board. Given this similar performance we now concern ourselves with the impact of the price predictor for each model. Since we are trying to model price sensitivity in customers, a larger impact of the price predictor is desired. To quantify this impact, the following analysis are made.
\subsubsection{Impact of the standardized coefficient}
We first analyze the impact of the standardized coefficient for the price predictor (for clarity denoted as $\beta_{\text{price}}$) with respect to the rest of standardized predictors. That is $|\beta_\text{price}| / \sum^k_{j=1} |\beta_j|\times 100$.

It is worth noting that since the Combined Model is composed of 48 linear models, the representation of the aforementioned model is composed of 48 standardized coefficients. Figure \ref{Coefficient_Percentage} shows the comparison between the models.
\begin{figure}[H]
    \centering
    \includegraphics[width=0.8\textwidth]{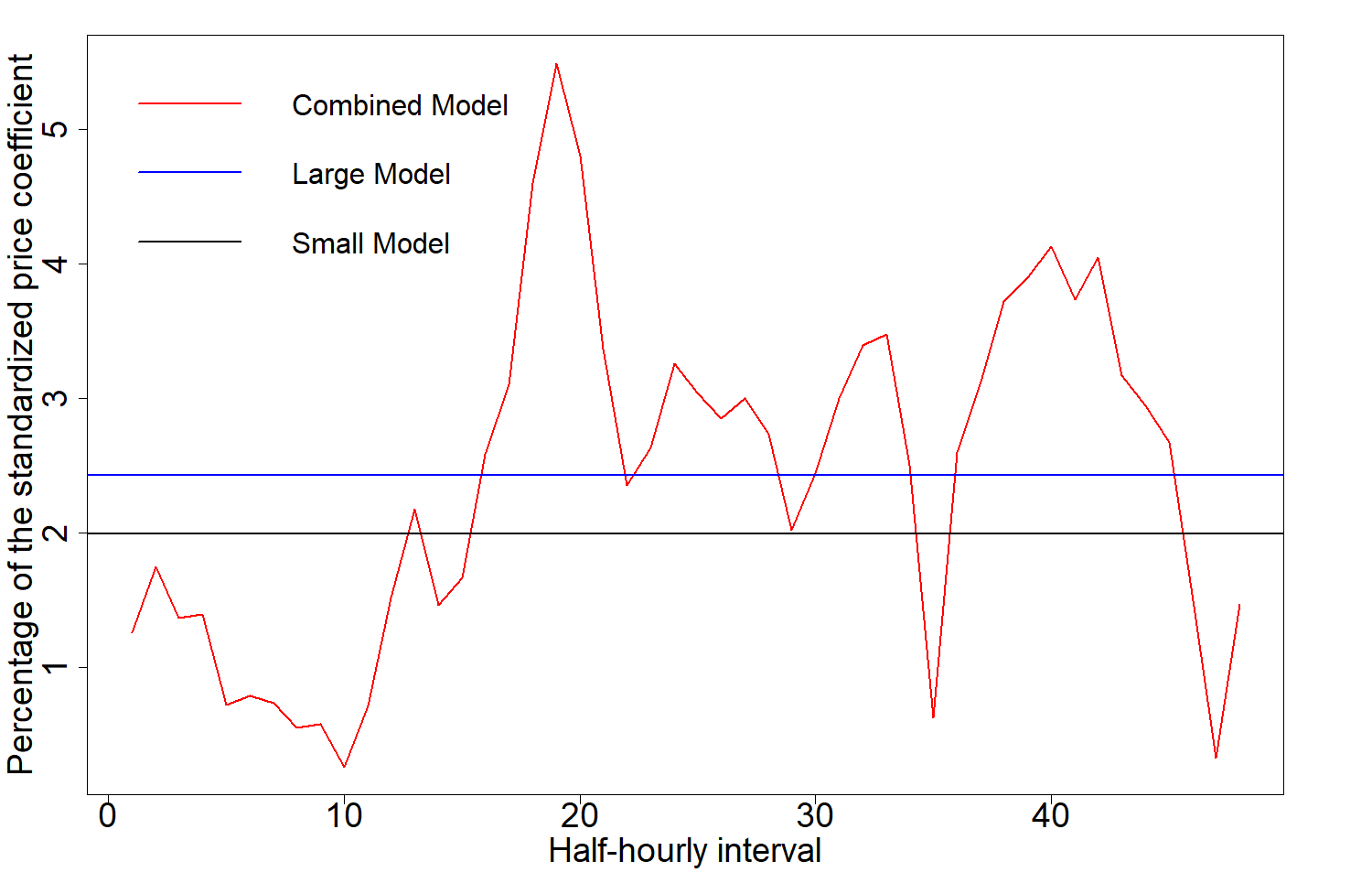}
    \caption{Plot of the relative percentage of the standardized price coefficient with respect to the rest for each model.}
    \label{Coefficient_Percentage}
\end{figure}
It is clear from the figure that the impact of the price predictor through this metric is low for all three models. It is interesting to compare the Large and Small models and note that even though the Large Model has around 3 times as many predictors, the percentage taken up by the standardized price predictor is larger than that of the Small Model. Upon further inspection this is due to some of the lag predictors taking much higher values than in the Large Model.
\subsubsection{Changes in RMSE and MAE when dropping the price predictor}
To further explore the impact of the price predictor we now investigate the impact in the models' performances when the price predictor is removed and the models are fitted again. We compare the differences in RMSE and MAE for the models with and without the price predictor across both the train and test sets. Once again the Combined Model's differences are represented for each half-hourly interval for the train set, however, for the test set the number of data points was deemed insufficient to represent differences for each half-hourly period, so this differences are averaged across all the linear models which compose the aforementioned model. The results are plotted in Figures \ref{RMSE_diff} and \ref{MAE_diff}.
\begin{figure}[H]
    \centering
    \includegraphics[width=0.8\textwidth]{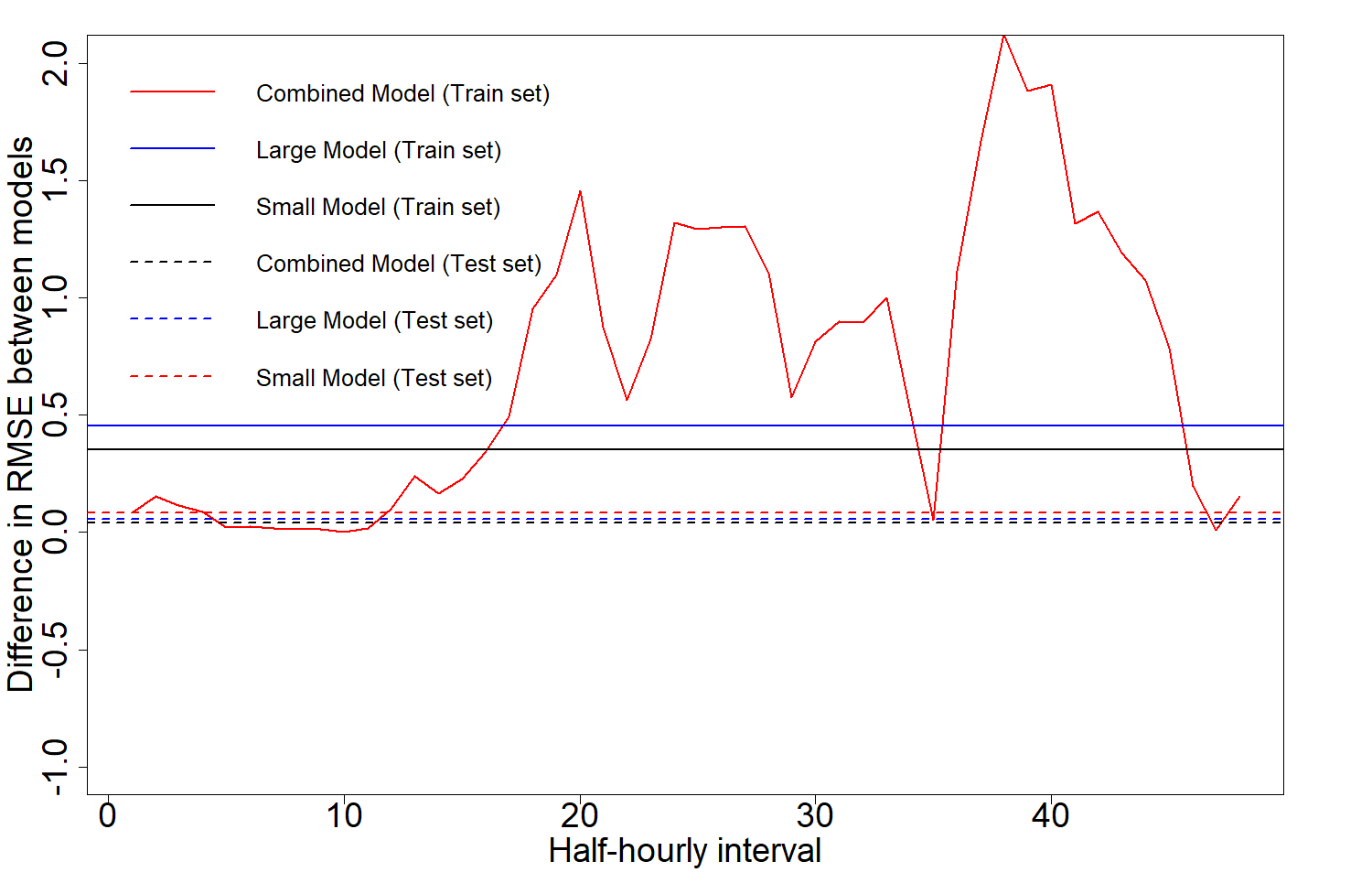}
    \caption{Plot of the difference in RMSE between models with and without the price predictor.}
    \label{RMSE_diff}
\end{figure}
\begin{figure}[H]
    \centering
    \includegraphics[width=0.8\textwidth]{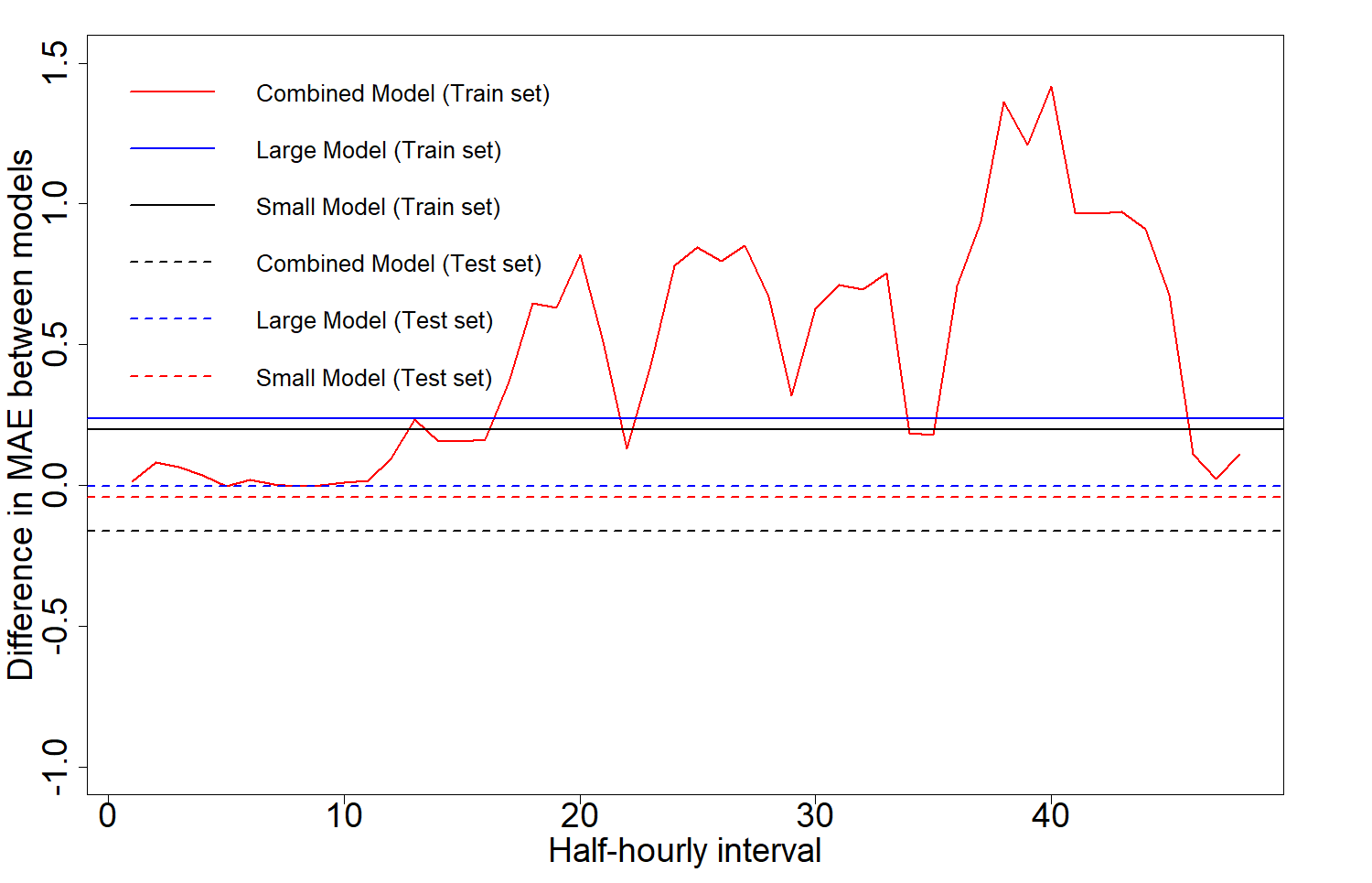}
    \caption{Plot of the difference in MAE between models with and without the price predictor.}
    \label{MAE_diff}
\end{figure}
The results show that the importance of the price predictor is low across all 3 models compared and some over fitting is apparent when comparing the respective performances between train and test sets. It is interesting to note that the only model which does not decrease its performance for neither RMSE and MAE nor train or test sets when the price predictor is dropped is the Large Model. Hence, we can conclude that the impact of the price predictor is low in the three models.

\subsubsection{Model selection}
Despite the results in the last section and considering that this low dependence on price probably stems from the shortcomings of the dataset exposed in Section \ref{dataset}, since choosing the best linear model possible is not the main focus of our work, we deem the price coefficient to be large enough for our purposes in all 3 models. Taking into consideration now both the metrics and the importance of the price predictor exposed in Section \ref{comparisons} we conclude that the best fit for our work is the Large Model. It does not exhibit as much over fitting behaviour as the Combined Model, while being simpler and easier to work with. Moreover, it exhibits a better performance and higher impact for the price predictor with respect to the Small Model, while maintaining all its benefits aside from a lower number of variables which was not deemed sufficient to justify its selection. Thus, from this point forward we consider only the Large Model and will be referred to simply as the linear model.

\subsection{Scenario generation and reduction}
For the stochastic problem dealt with in the following section we require a scenario set to optimize over. To this end we follow the procedure describe in Section \ref{scenario_gen} to generate a total of 1000 scenarios including random realization of pool prices $\bar{\lambda}^{P}_{t,\omega}$, availability of solar resources $A^{\text{PPA}}_{t,\omega}$ and price coefficient predictors $\beta_{\omega}$. However, utilizing this raw scenario set can be computationally prohibitive. Thus we seek to reduce it to a smaller subset of scenarios with similar statistical properties. To do this, we employ a scenario reduction technique proposed by \cite{Pineda2010ScenarioRF} where the distance between scenarios is measured in terms of their impact on the problem objective function. Hence, we solve deterministic versions of problem \eqref{risk_CVaR} for each original scenario. Afterwards, mean shift clustering is utilized to group the scenarios into clusters with center $C_k$, finally the reduced scenarios are obtained by finding the closest scenario to the centers of these clusters, that is if $\mathcal{S}$ is the unreduced scenario set, $K$ is the number of clusters produced by the mean shift algorithm, $\mathcal{S}_r$ is the reduced scenario set and $f$ is the cost function given by the profit in \eqref{risk_CVaR}:
\begin{equation}
    \mathcal{S}_r = \left\{ s \in \mathcal{S} \,| \, \text{argmin}\norm{f(s)-C_k}{}_2, \quad k \in \{1,\dots,K\} \right\}
\end{equation}
Once the reduced scenario set is produced, new probabilities have to be assigned to each reduced scenario by taking into account how many unreduced scenarios originally belonged to the cluster of the reduced scenario. This whole procedure is represented by the CDF's of the reduced and unreduced scenario sets of Figure \ref{CDF}.
\begin{figure}[H]
    \centering
    \includegraphics[width=\textwidth]{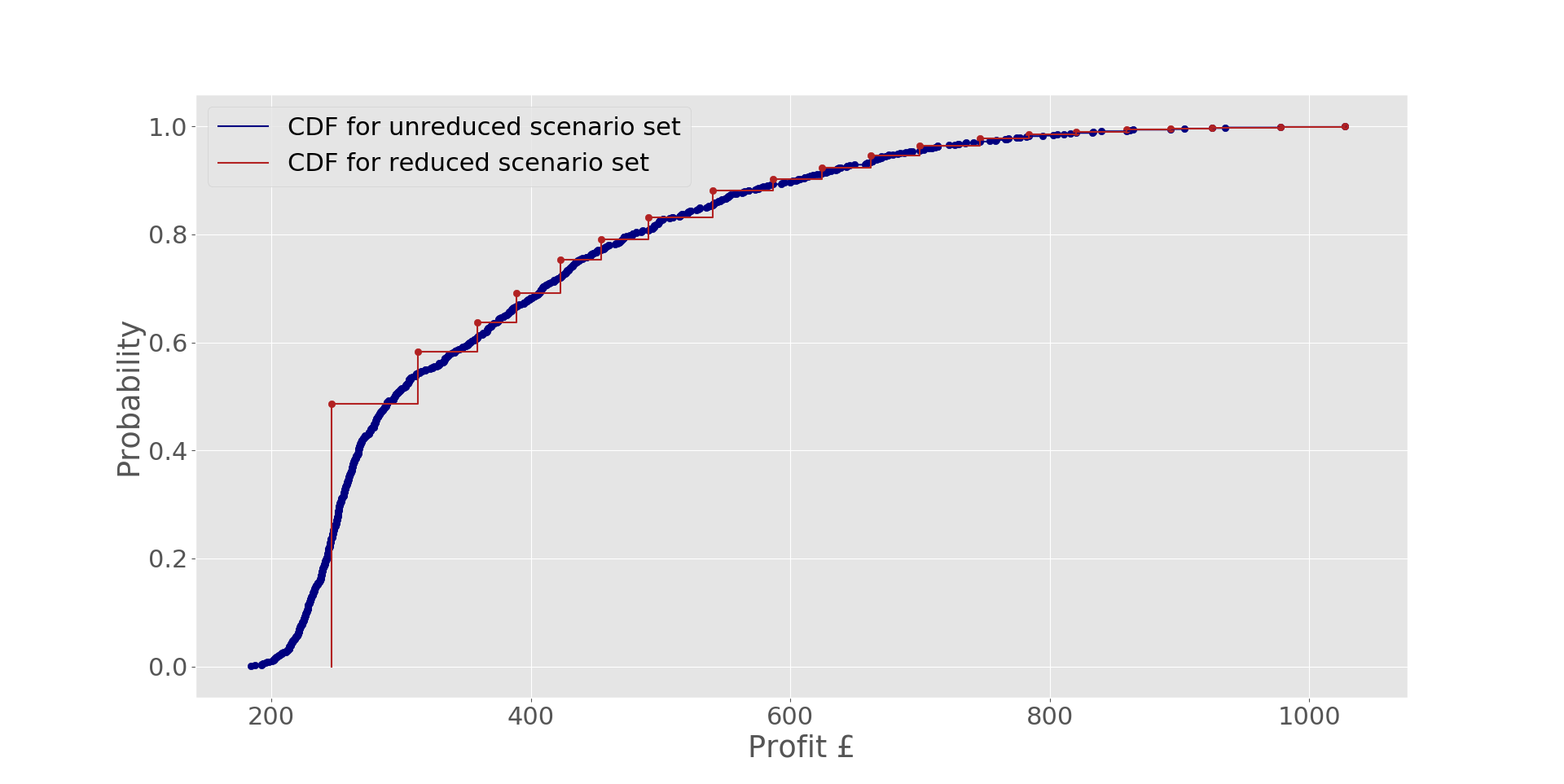}
    \caption{CDF's of reduced and unreduced scenario sets.}
    \label{CDF}
\end{figure}

\subsection{Solution to the Stochastic Problem}
We solve problem the stochastic problem (\ref{risk_CVaR}) by means of the scenario set described in the previous section. Moreover, we consider that the maximum energy available for purchase from the forward and PPA contracts ($P^{\rm B}_{\rm max}$ and $P^{PPA}$) is 80 kWh, at purchase prices of 4.6 and 4.8 p/kWh, respectively.
The mean value of the variable part of the price ($\bar{\lambda}^E$) is assumed to be 1 p/kWh. Bounds on the variable part of the price ($\lambda^E_t$) are impose assuming $\gamma=0.25$ and the CVaR is computed over the first 10-percentile of the profit distribution ($1-\alpha=0.9$). Figure \ref{Pool_solar} depicts the summary (mean values) of the scenarios for pool prices and solar availability considered (generated as described in Appendix \ref{Anexo:Scenario}). The dotted line is the mean of pool prices over the half-hourly time periods which have a solar availability of at least $0.1$. Although this value was chosen arbitrarily, careful examination of the figure makes it clear that it will always be greater than the mean over all half-hourly time periods.

\begin{figure}[H]
    \centering
    \includegraphics[width=\textwidth]{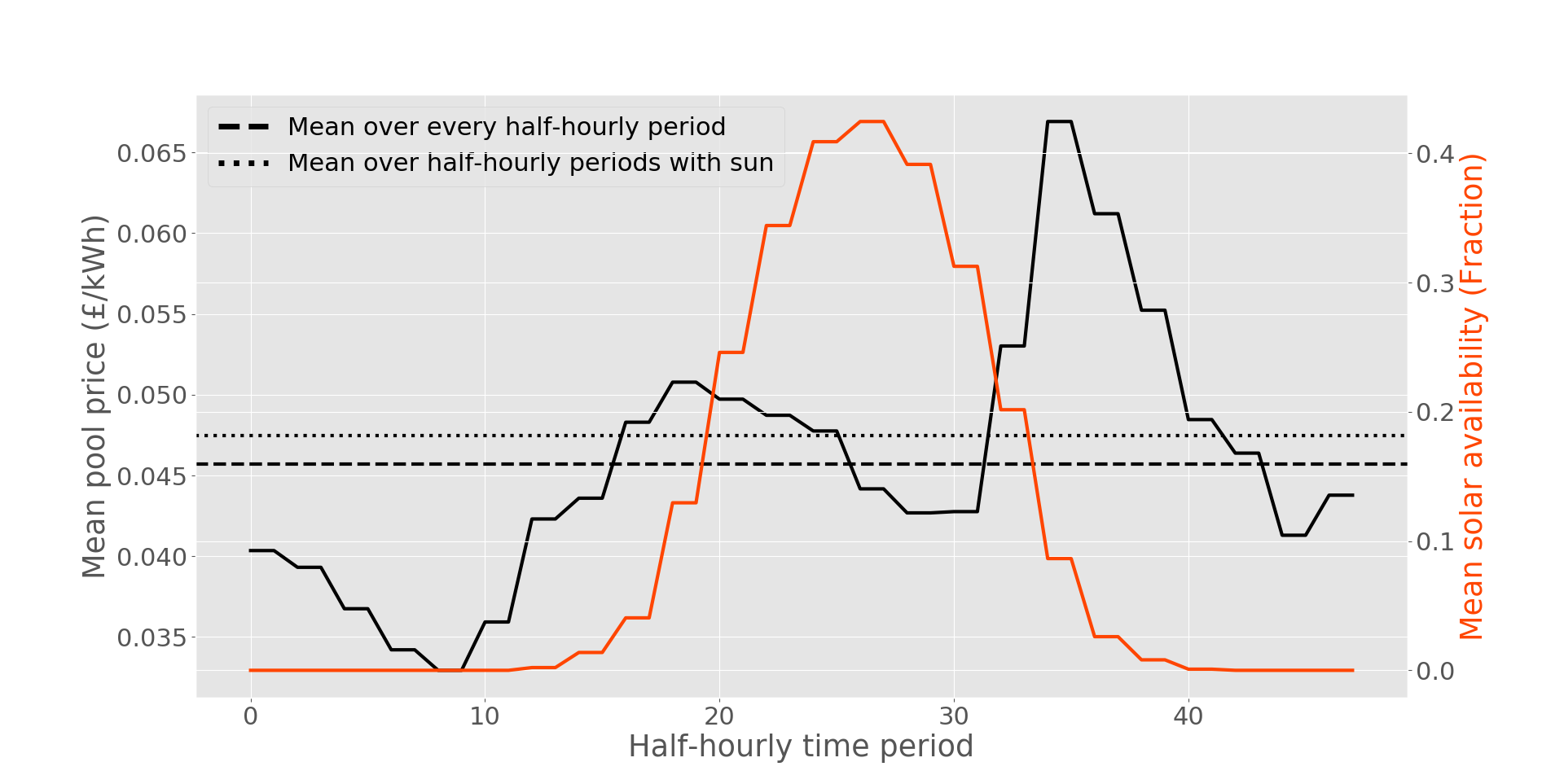}
    \caption{Mean pool price and mean solar availability for the test month (December).}
    \label{Pool_solar}
\end{figure}

\subsubsection{Efficient frontier}
We are first interested in deciding which values of $\chi$ are meaningful to explore. For that purpose we compare, for different values of $\chi$, the expected profit versus the CVaR. Note that values of $\chi = 0$ (risk neutral) and $\chi = 1$ (risk averse) are approximated by $\chi = 0.0001$ and $\chi = 0.9999$ respectively to avoid numerical inconsistencies.
\begin{figure}[H]
    \centering
    \includegraphics[width=\textwidth]{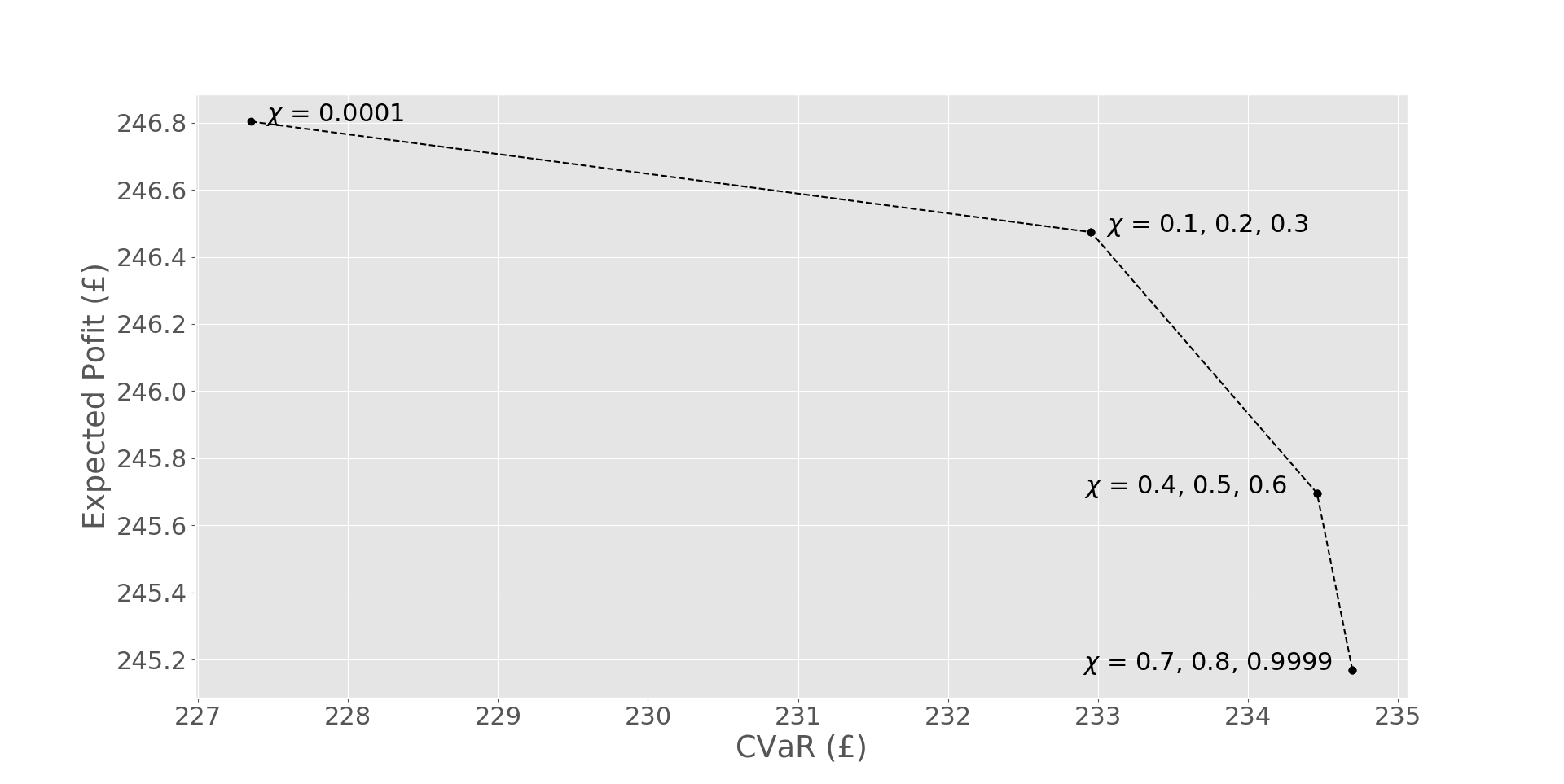}
    \caption{Expected profit versus CVaR for different values of $\chi$. Problem was solved for values between 0 and 1 at intervals of $0.1$. This type of plot is called efficient frontier.}
    \label{frontier}
\end{figure}

The resulting efficient frontier is depicted in Figure \ref{frontier}, where we found 4 main groups of values of $\chi$ which entail different combination of expected profit vs CVaR. For clarity, in the following section we focus our analysis on the 2 extreme values of $\chi = 0.0001$ and $\chi = 0.9999$. Also, it is interesting to note that the losses in expected profit are higly compensated by the increase of the CVaR, thus we can infer that the inclusion of risk for the retailer's optimal decision is probably desired, especially a value of $\chi$ between 0.1 and 0.3 since that is the region where the drop in expected profit is small with a large increased of the CVaR.

\subsubsection{Daily Price distribution}
We want to study how the prices offered by the retailer ($\lambda^R_{t,\omega}$) behave at each hour and how these vary among scenarios.  Figures \ref{fig:price_chi0} and \ref{fig:price_chi1} represent the distribution of prices for each day of the week in terms of their average half hourly values over the test set (month of December), for the risk neutral and risk averse cases, respectively. The maximun and minimun values are also included to quantify the level of price variability.
\begin{figure}[H]
  \centering
  \includegraphics[width=.8\linewidth]{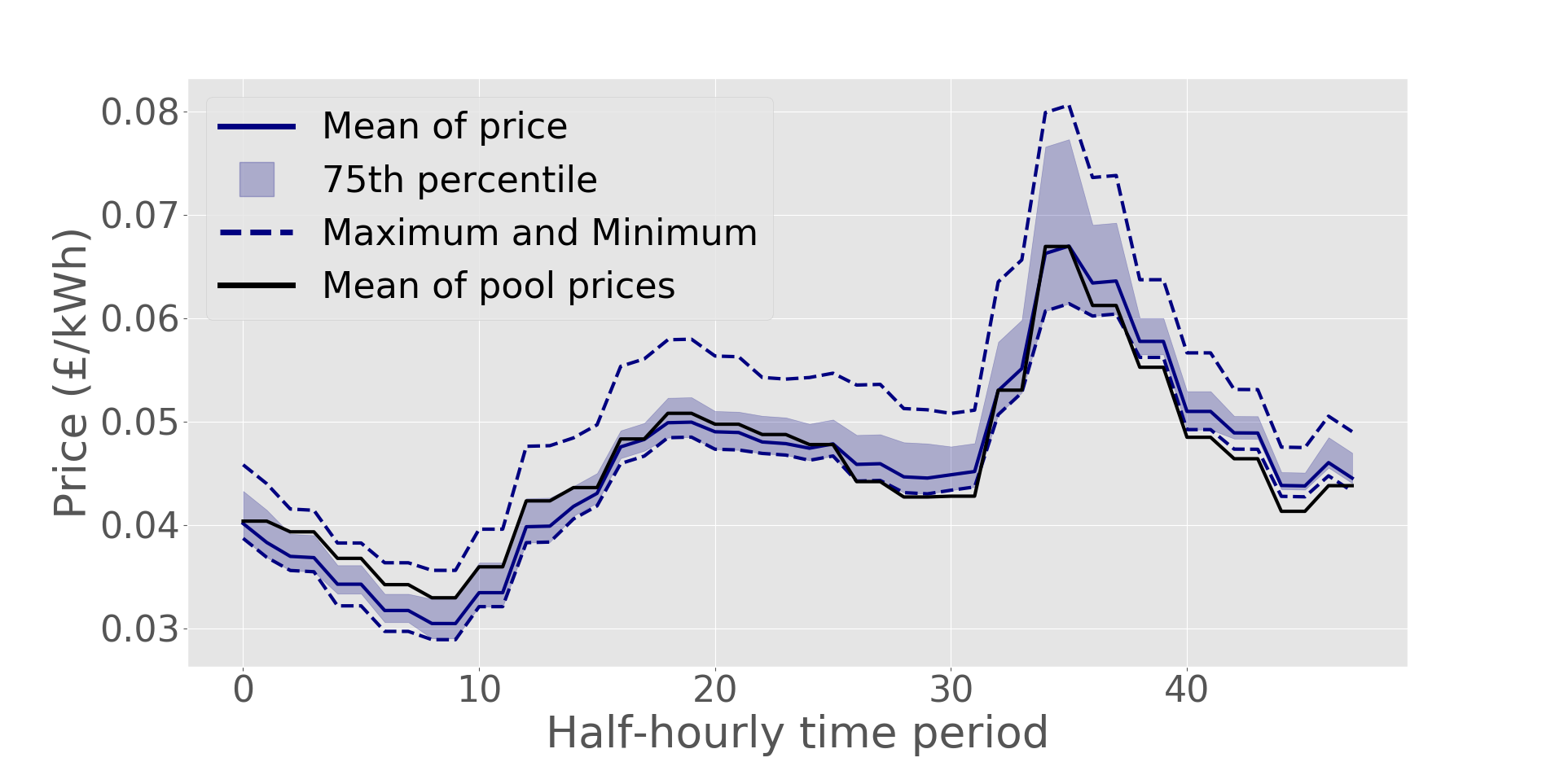}
  \caption{Price distribution for $\chi = 0.0001$ (risk-neutral)}
  \label{fig:price_chi0}
\end{figure}

\begin{figure}[H]
  \centering
  \includegraphics[width=.8\linewidth]{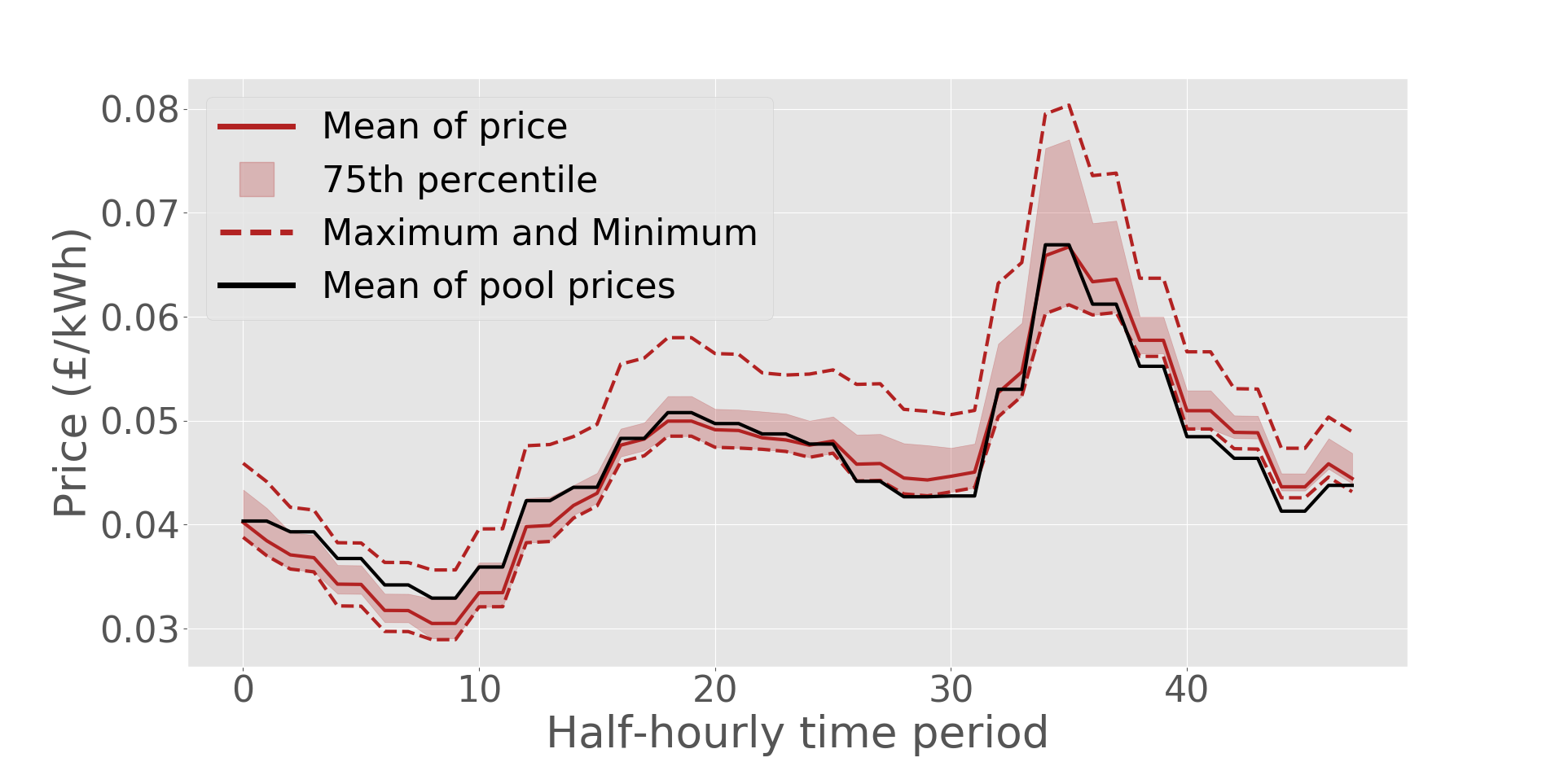}
  \caption{Price distribution for $\chi = 0.9999$ (risk-averse)}
  \label{fig:price_chi1}
\end{figure}

Figures \ref{fig:price_chi0} and \ref{fig:price_chi1} show that the price distribution is the same for the two risk aversion levels considered. This is caused by the tight constrains imposed to control the price and indexing it to the pool price. Indeed we can observe how the pool and the retail price follow analogous trajectories. Note that for some periods, the retailer is able to offer tariffs on average below the spot price. This is because part of its energy has been already purchased at fixed prices (below the spot price at many hours) through forward based or PPA contracts. The resulting total demand from the consumers (Figure \ref{Pool_demand}) is also the same for the two risk aversion levels considered. It follows the same trend that the retail price, as it linearly depends on it. 

\begin{figure}[H]
    \centering
    \includegraphics[width=0.8\textwidth]{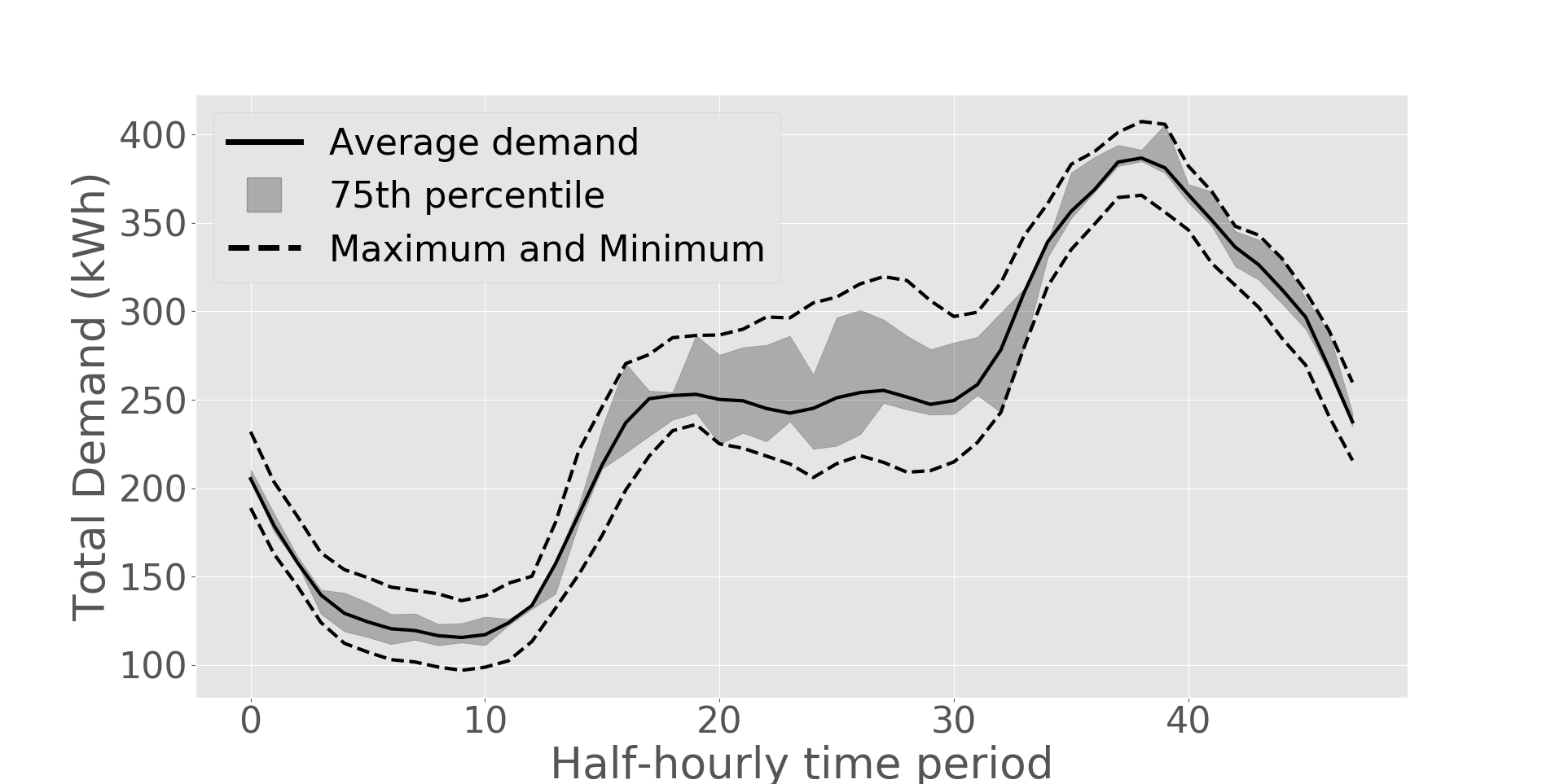}
    \caption{Energy demand.}
    \label{Pool_demand}
\end{figure}

To exemplify this effect, we remove from the problem constraints \eqref{PCPV_sto_08} and \eqref{PCPV_sto_09}, which represent market regulation to control retail prices, and we rerun the simulation to regenerate Figures \ref{fig:price_chi0_freeprice} and \ref{fig:price_chi1_freeprice}.

\begin{figure}[H]
  \centering
  \includegraphics[width=0.8\linewidth]{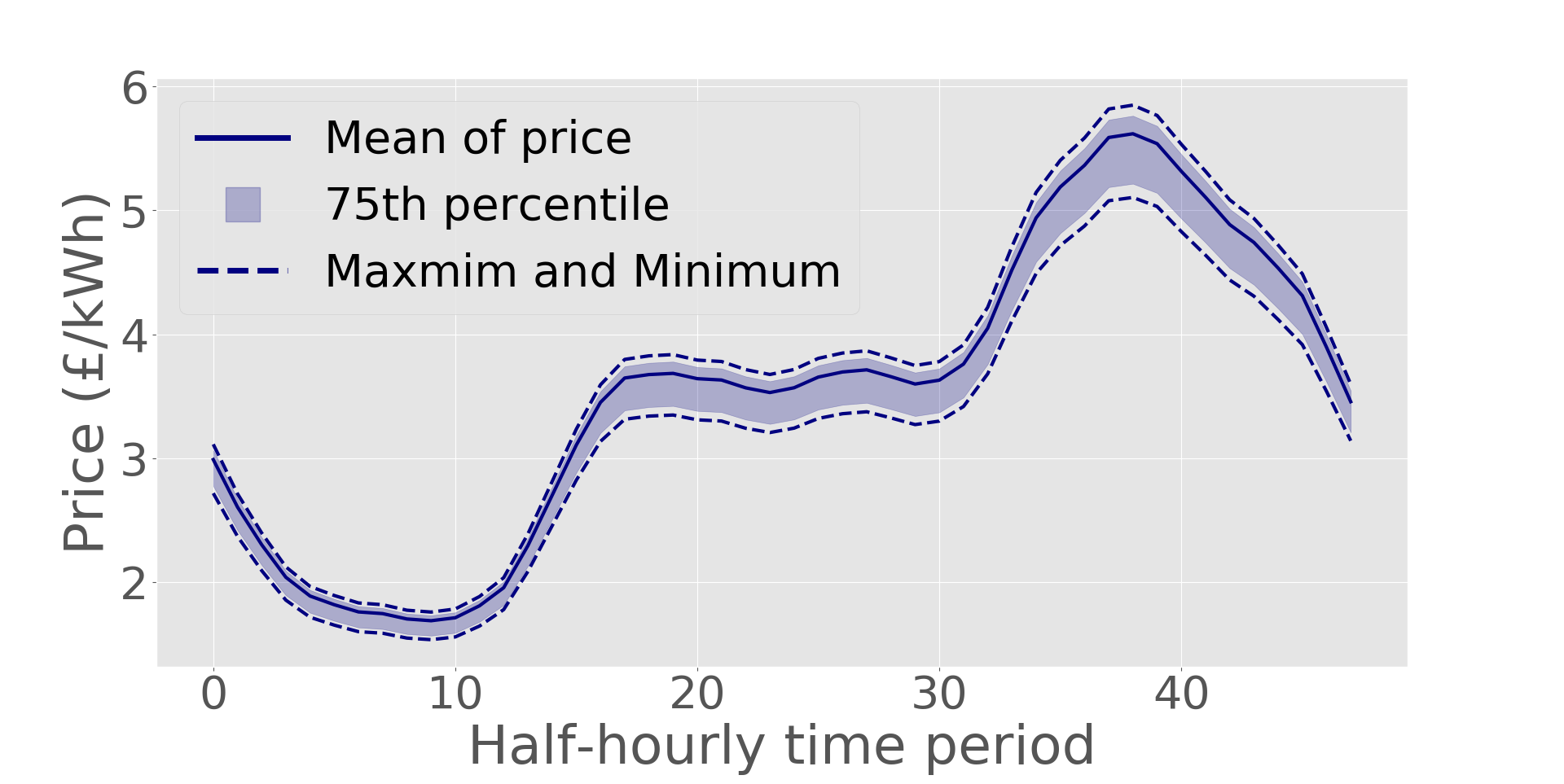}
  \caption{Price distribution for $\chi = 0.0001$ (risk-neutral), no direct constraints on retail prices.}
  \label{fig:price_chi0_freeprice}
\end{figure}
\begin{figure}[H]
  \centering
  \includegraphics[width=0.8\linewidth]{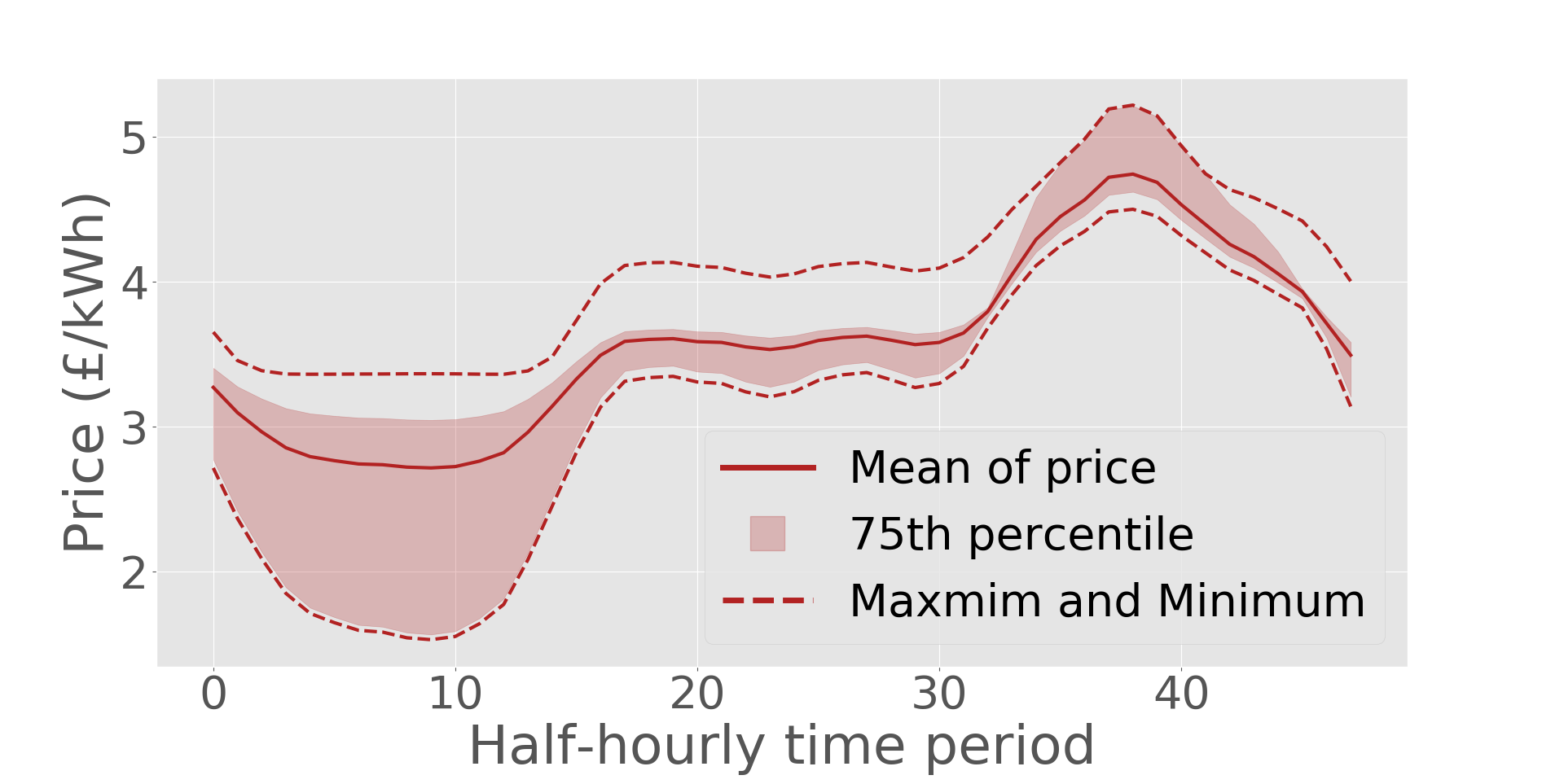}
  \caption{Price distribution for $\chi = 0.9999$ (risk-averse), no direct constraints on retail prices.}
  \label{fig:price_chi1_freeprice}
\end{figure}
The prices are now unrealistically high as the retailer would have a very strong position in the market to raise prices (market power). However, an interesting results is that the retailer would have sufficient control over the prices for it to be an effective variable to control risk. This behaviour would be expected with groups of consumers with enough price-elasticity so that external price regulation (constraints) are no longer needed. This particular aspect is studied in detail in Section \ref{beta_shift}.

In the following we analyze how risk is effectively controlled by the retailer through the appropriate involvement in the different types of contracts that are available.

\subsubsection{Profit distributions}
The profit cumulative distribution functions (CDFs) over the scenarios for the different values of $\chi$ are represented in figure \ref{fig:CDF}.
\begin{figure}[H]
    \centering
    \includegraphics[width=\textwidth]{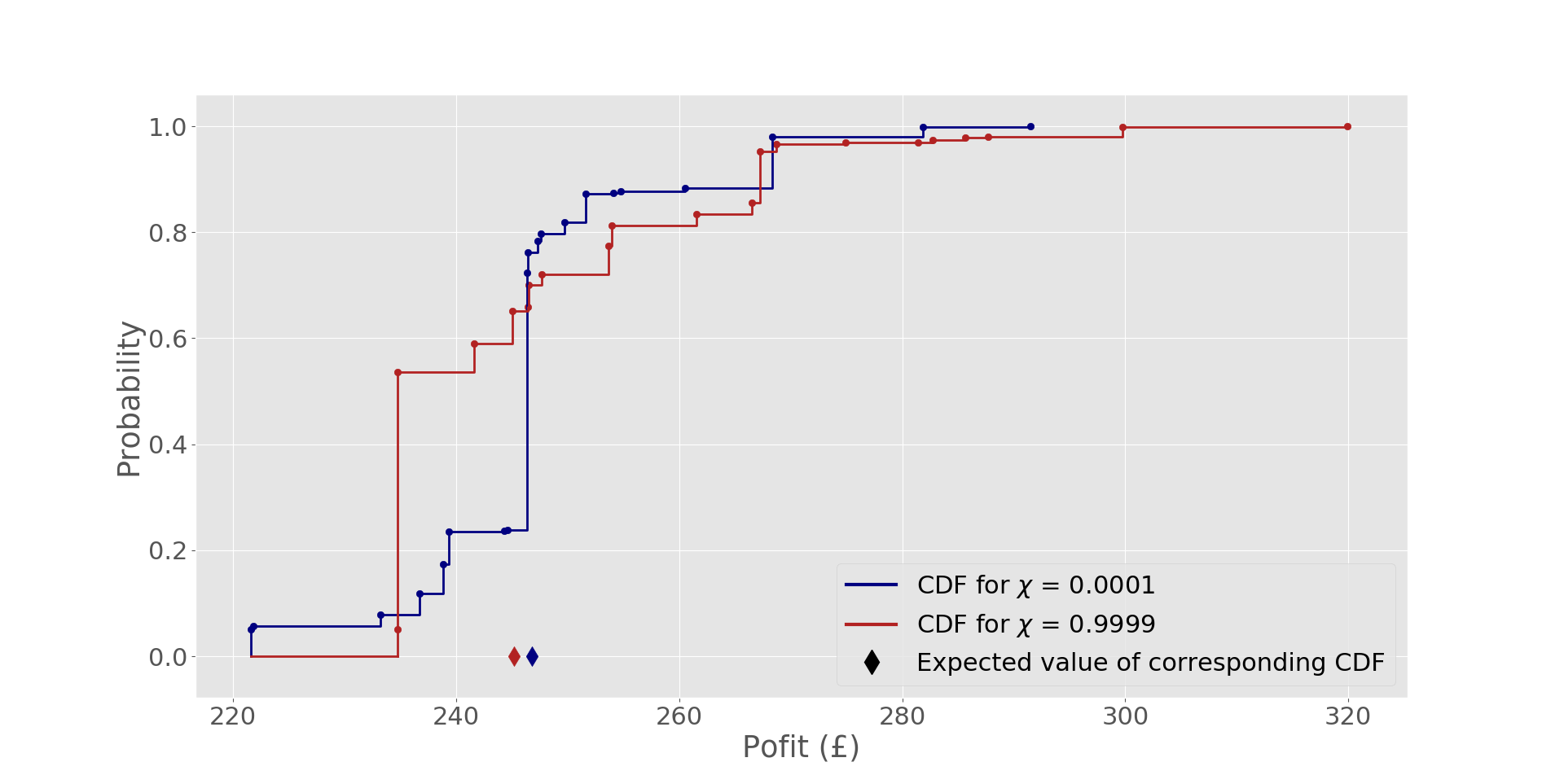}
    \caption{CDF of the profit distributions for $\chi = 0.0001$ and $\chi = 0.9999$}
    \label{fig:CDF}
\end{figure}
As expected, when risk aversion is considered, the left tail of the profit distribution moves to the right but the expected profit is reduced. However, it is interesting to see that a side effect of moving the left tail to the right also increases the profits in some of the scenarios to higher values than any of the profits of the distribution for $\chi = 0.0001$. This is surprising because usually, when risk becomes a factor, the CDF of said distribution compresses over the x axis with respect to the CDF without risk considerations. Nonetheless, there is nothing preventing this behaviour from arising as the CVaR only restricts the left tail of the profit distribution.

\subsubsection{Impact on contract prices}
We now investigate the decisions the retailer makes involving the amount of energy purchased from both the forward base contract and the PPA contract ($p^{\rm B}$ and $p^{\rm PPA}$) depending on the prices of said contracts ($\bar{\lambda}^{\rm B}$ and $\bar{\lambda}^{\rm PPA}$), and the risk aversion level. The results are presented in figures \ref{fig:pb_chi0}, \ref{fig:pppa_chi0}, \ref{fig:pb_chi1} and \ref{fig:pppa_chi1} as the percentage of energy bought with respect to the maximum available of the contract analyzed. Darker color means higher percentage. The actual percentage is represented within each cell.

For the risk neutral case (Figure \ref{fig:pb_chi0}), the forward base contract $p^{\rm B}$ is only attractive if its price is below 4.59 p/kWh. In this case it is fully contracted regardless of the price offered by the renewable-based PPA contract $p^{\rm PPA}$. Similarly, the PPA contract is signed at is maximum capacity if price is below 4.81 p/kWh (Figure \ref{fig:pppa_chi0}), regardless of the price offered by the forward base contract. This threshold price is above the previous one because the PPA contract with a solar PV producer provides energy to the retailer in the central part of the day, which has associated higher pool prices (see Fig. \ref{Pool_solar}). On the other hand, the forward contract is only competitive at lower prices because is a base contract which provides energy in every period of the day, including night periods with low pool prices.  

\begin{figure}[H]
    \centering
    \includegraphics[width=0.7\textwidth]{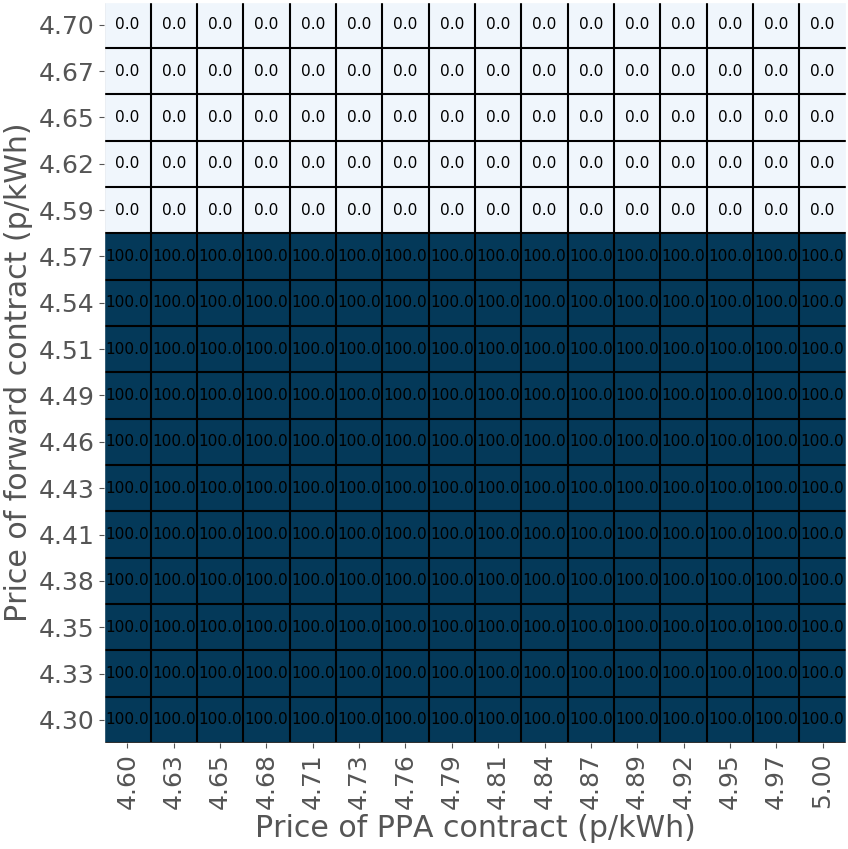}
    \caption{Percentage of $p^{\rm B}$ contracted depending on prices of contracts for $\chi = 0.0001$ (risk-neutral).}
    \label{fig:pb_chi0}
\end{figure}
\begin{figure}[H]
    \centering
    \includegraphics[width=0.7\textwidth]{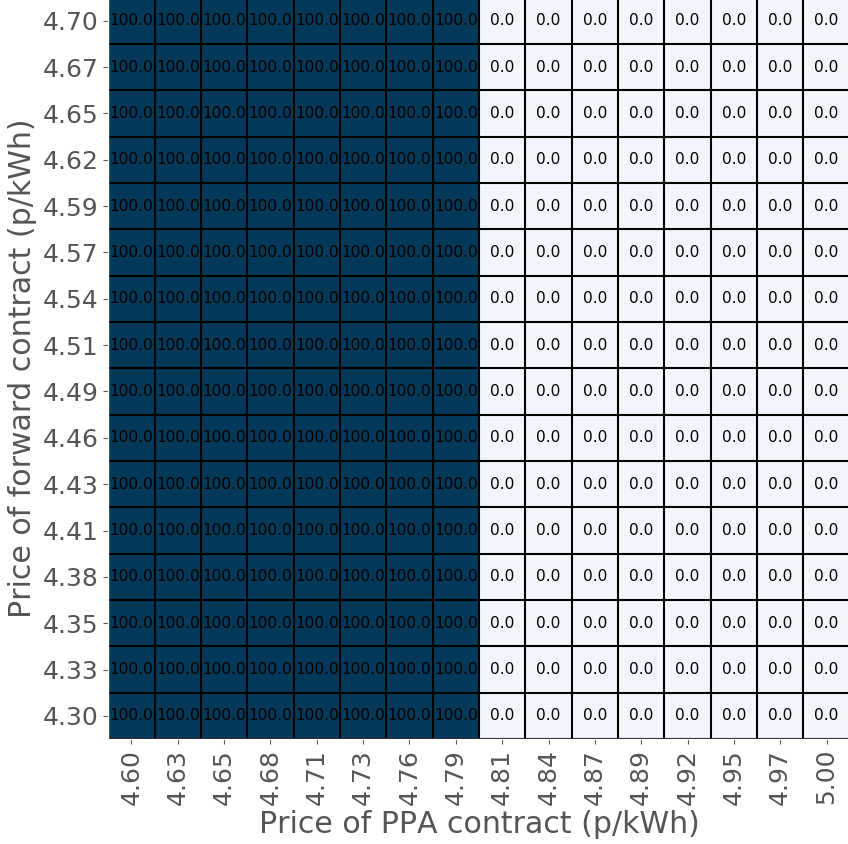}
    \caption{Percentage of $p^{\rm PPA}$ contracted depending on prices of contracts for $\chi = 0.0001$ (risk-neutral).}
    \label{fig:pppa_chi0}
\end{figure}

In the risk averse case, the forward base contract $p^{\rm B}$ (Fig. \ref{fig:pb_chi1}) is signed at its maximum capacity if its price is below 4.43 p/kWh, which is below the price-threshold observed in the risk neutral case. However, for prices above this value, this contract is also partially contracted if the prices of the PPA are relatively high. Similarly, the PPA contract (Figure \ref{fig:pppa_chi1}) is signed at is maximum capacity if its price below 4.63 p/kWh, which also below the price-threshold observed in the risk neutral case. We can also observe that for prices above 4.63 p/kWh, the PPA is also contracted (partially) when the forward contract price is high. This diversification strategy, which is characteristic of risk averse settings, is more present for PPA contracting, as it introduces further levels of uncertainty due to its dependency on the renewable source. Overall, the risk-averse forward and PPA contracting volumes are decreased and used only if the prices are sufficiently attractive (small). This strategy helps avoiding potential unfavourable scenarios where forward contracting may exceed the demand realization, and hence incur in low profits.

\begin{figure}[H]
    \centering
    \includegraphics[width=0.7\textwidth]{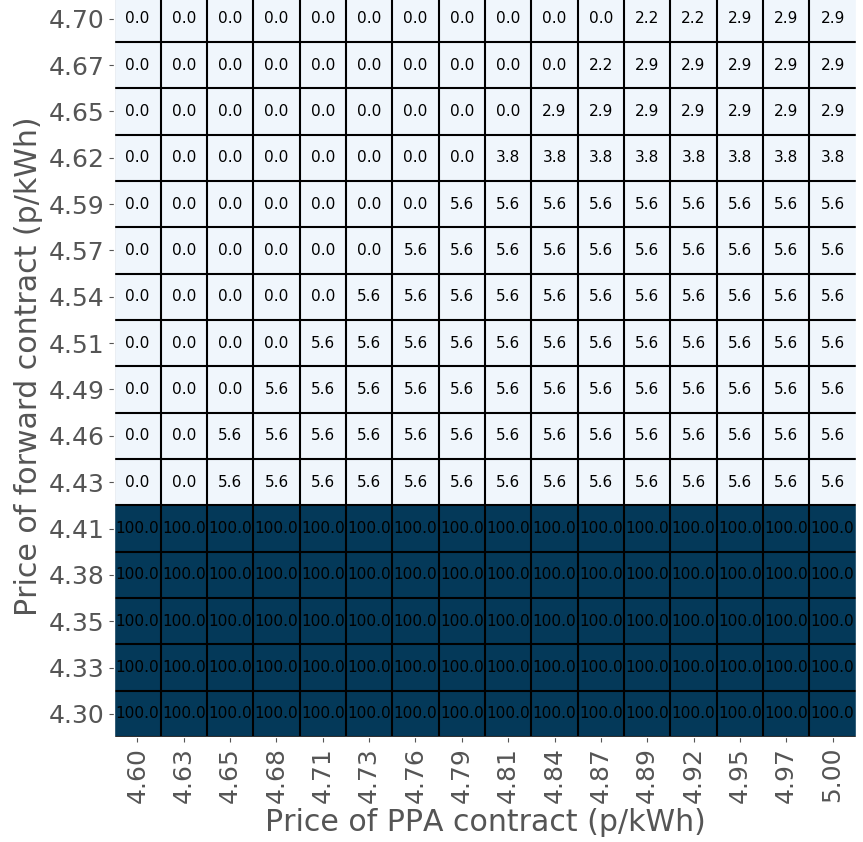}
    \caption{Percentage of $p^{\rm B}$ contracted depending on prices of contracts for $\chi = 0.9999$ (risk-averse).}
    \label{fig:pb_chi1}
\end{figure}
\begin{figure}[H]
    \centering
    \includegraphics[width=0.7\textwidth]{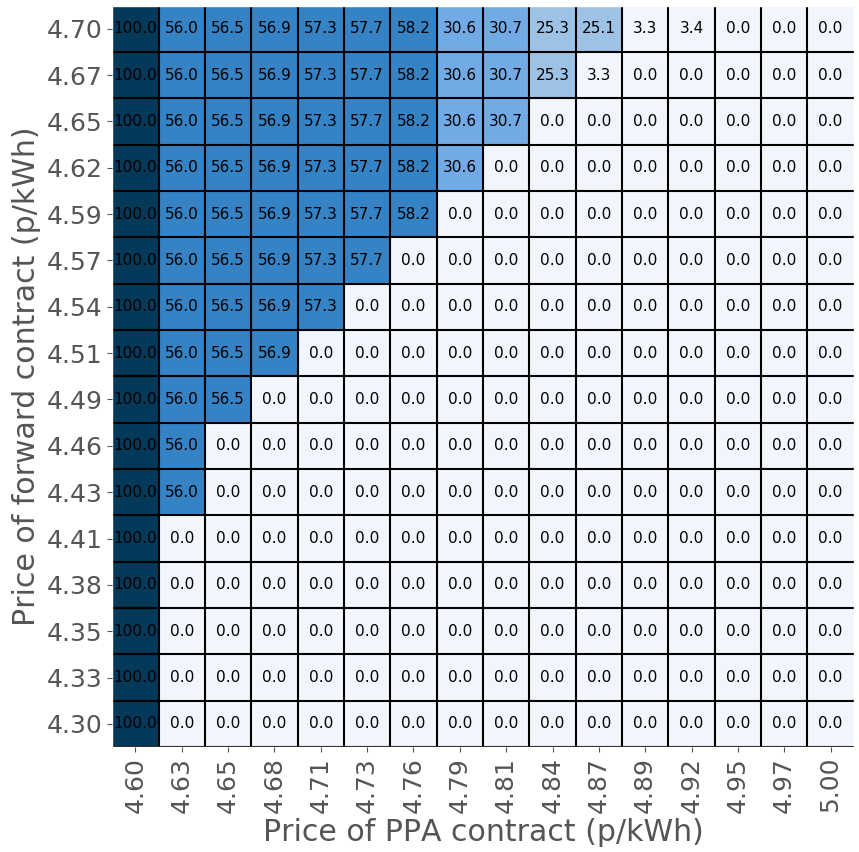}
    \caption{Percentage of $p^{\rm PPA}$ contracted depending on prices of contracts for $\chi = 0.9999$ (risk-averse).}
    \label{fig:pppa_chi1}
\end{figure}

\subsubsection{Effects of shifts on the mean of the price coefficient distribution}\label{beta_shift}

We now tackle the effects of modifying the price coefficient to simulate more sensitive costumers to price signals. We seek to recreate the actual tendency of electrical systems where consumers are becoming more and more elastic to prices thanks to the technological advances provided within the smart grid paradigm (smart-meters, electric-vehicles and local electricity storage, intelligent appliances, distributed generation, etc.), and its impact on retailers' strategies.

For this purpose we modify the distribution of $\beta_1$ (assuming again this represents the price coefficient) given by \eqref{price_dist} by adding a shift term $\beta_{\text{shift}}$ as in:
\begin{equation}
    \hat{\beta}_1 \to \mathcal{N}\left(\mathbb{E}[\hat{\beta}_1] - \beta_{\text{shift}}, \sigma^2_\epsilon\sum_{j=1}^n[(\mathbf{X}^T\mathbf{X})^{-1}\mathbf{X}^T]^2_{1j}\right)
\end{equation}
Note that the variance of the distribution is kept the same and only the mean is shifted. This would be equivalent to having a hypothetical dataset with a greater impact of the price coefficient while maintaining the accuracy of the estimation of the coefficient's distribution. We now consider 2 effects, the impact on the average price \eqref{eq:average_price}, taken to be a scalar summary of the price distribution, and the impact on the profit distribution over the scenario set.
\begin{equation}\label{eq:average_price}
    \text{average price} = \frac{1}{|T|}\sum_{t \in T}\sum_{\omega \in \Omega}\pi_{\omega}\lambda^{\rm R}_{t,\omega}
\end{equation}

The first result is that the impact on the average price as a function of $\beta_{\text{shift}}$ under the current formulation \eqref{risk_CVaR} is almost negligible, as it remains constant at a value of for both risk-neutral and risk-averse cases of 4.56 p/kWh. This is caused by the tight constraints that regulate the retail price. However, the effects on the profit distribution are more interesting and are represented as violin plots in Figure \ref{fig:violin}. It is clear that as $\beta_{\text{shift}}$ increases (higher price-sensitivity of consumers), the expected profit decreases, albeit slightly. On the other hand, the distribution does not change as a function of $\beta_{\text{shift}}$ on this problem, as retail prices do not vary the main cause for this decrease is a reduction in the overall consumers demand.
\begin{figure}[H]
    \centering
    \includegraphics[width=\textwidth]{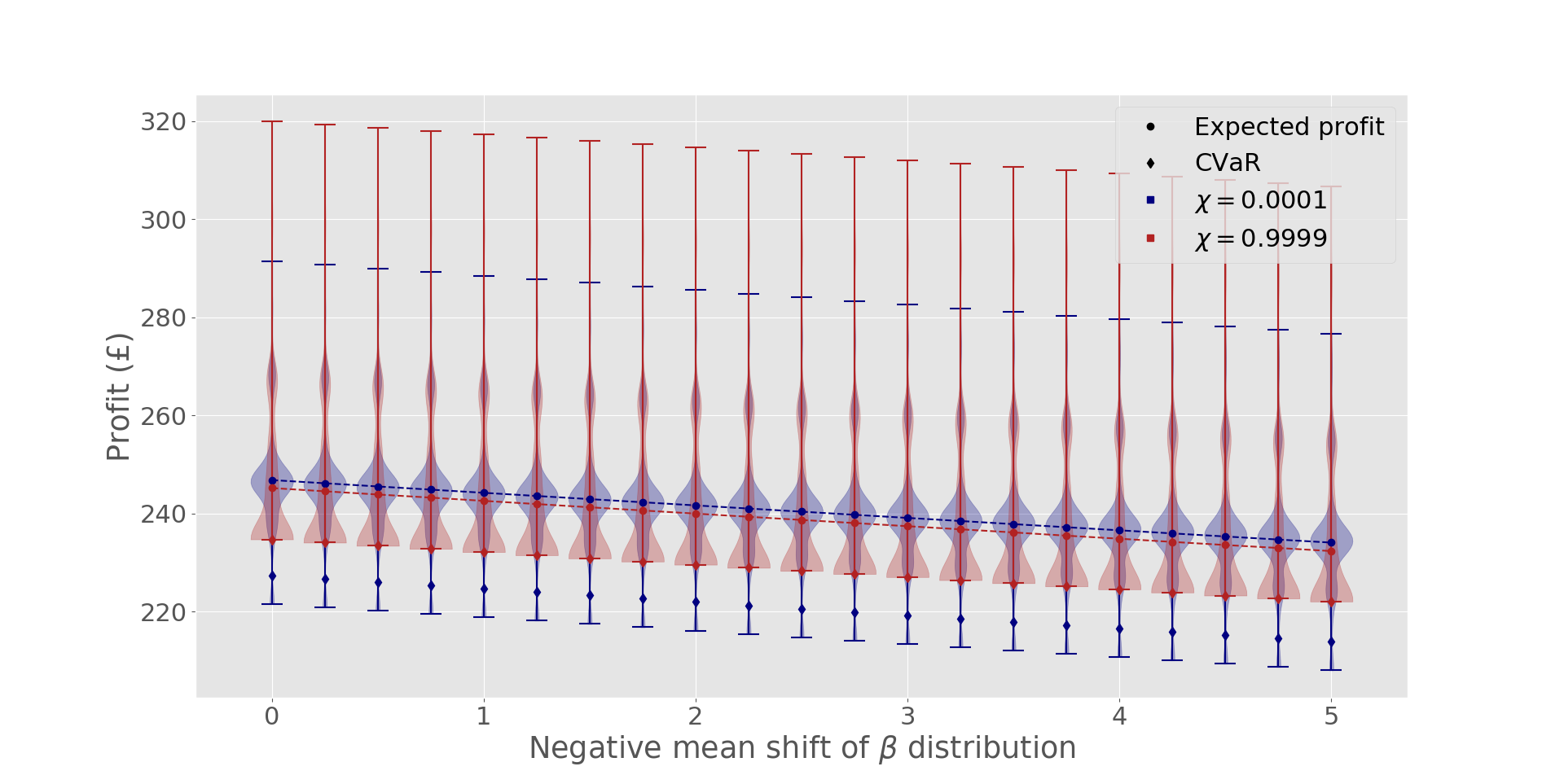}
    \caption{Violin plots showing profit distributions for both values of $\chi$.}
    \label{fig:violin}
\end{figure}

However, now that we are essentially increasing the impact of price-demand sensitivity ($\beta_1$), we can explore if for larger values of this parameter the retail price can be self-regulated by consumers and no external market rules are needed, i.e., verify if constraints \eqref{PCPV_sto_08}-\eqref{PCPV_sto_10} could be dropped from the model. For this purpose we first analyze for which values of $\beta_{\text{shift}}$ the self-regulation of prices is sufficient (keeping in mind that $\mathbb{E}[\hat{\beta}_1] = -0.36$). 
We also choose a `reasonable' boundary for the price by considering the mean of the actual pool prices from the dataset, which is 13.6 p/kWh, and by assuming that prices has to be lower than double this price.

Figure \ref{fig:prices} represents how the average retail price varies with $\beta_{\text{shift}}$ for risk-neutral and risk-averse retailers when solving problem \eqref{risk_CVaR} where the price-regulating constraints \eqref{PCPV_sto_08}-\eqref{PCPV_sto_10} are dropped. We see how retail prices self-regulate as the impact of $\beta_1$ increases (higher degree of price-sensitivity by consumers). Note that prices start being ``reasonable'' around a value of $\beta_{\text{shift}} \approx 4.75$. 
\begin{figure}[H]
    \centering
    \includegraphics[width=\textwidth]{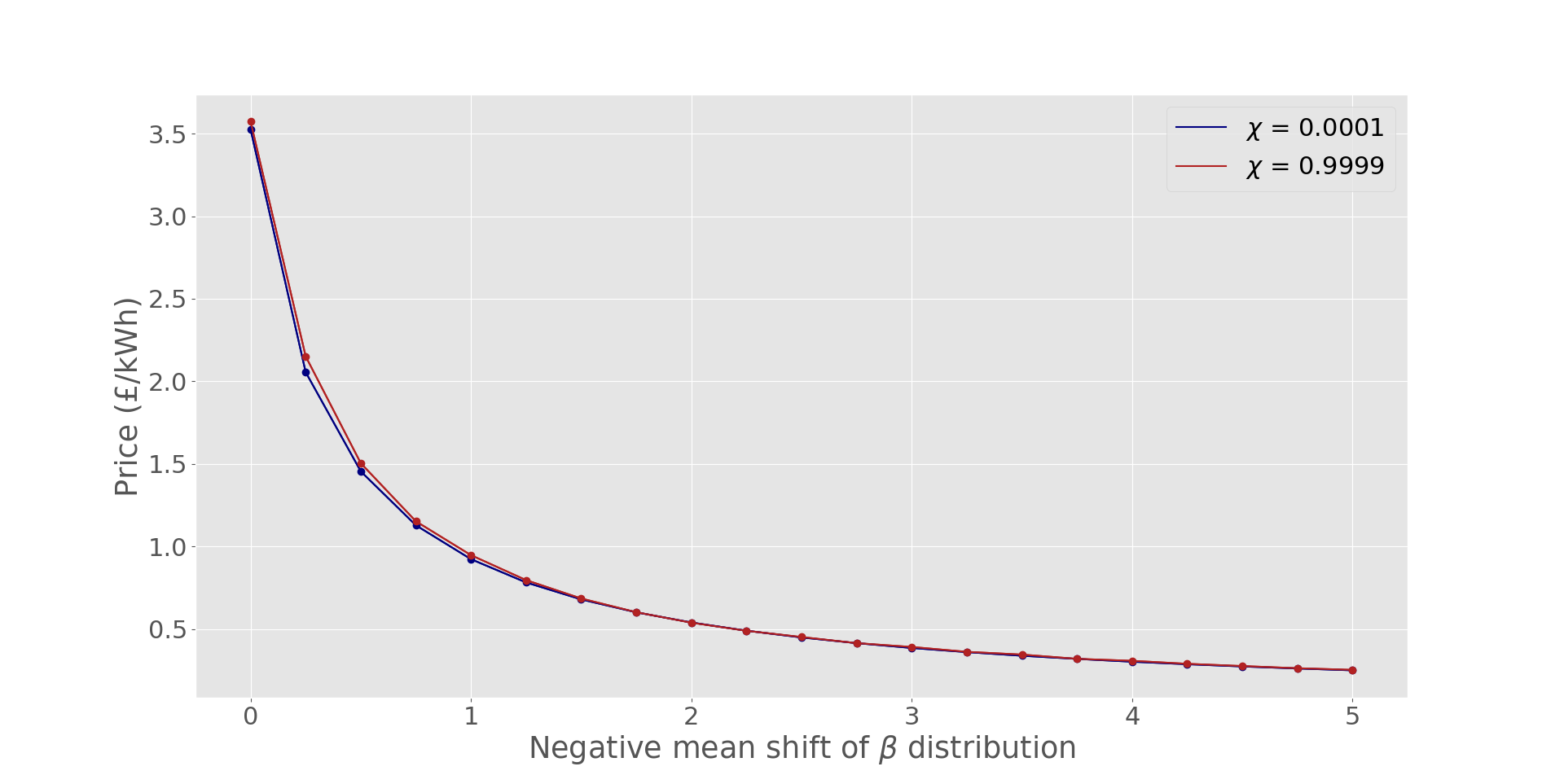}
    \caption{Plot of average price for both values of $\chi$.}
    \label{fig:prices}
\end{figure}

Having identified appropriate values of $\beta_{\text{shift}}$ to self regulate the retail prices, we now explore the profit distribution for the problem with equations (\ref{PCPV_sto_08})-(\ref{PCPV_sto_10}) dropped. As can be seen in Figure \ref{fig:profit_dist_free} the expected profits exhibit a similar behaviour to the retail prices in Figure \ref{fig:prices}. However the behaviour is much richer than that of Figure \ref{fig:violin} as not only does the profit vary non-linearly, but, as can be seen in subfigures \ref{fig:sub2_2} and \ref{fig:sub3} the profit distribution also varies (we illustrate the cases where $\beta_{\text{shift}}$ equals 0, 4.75 and 5). It might be tempting to compare the profits from figure \ref{fig:profit_dist_free} to those of \ref{fig:violin}, however the average price for $\beta_{\text{shift}} = 5$ is still around 25 p/kWh (now it does depend on $\chi$) while, as mentioned previously, average price for the restricted problem was 4.56 p/kWh. To mimic this price, $\beta_{\text{shift}}$ must be increased all the way up to 52 where the average price differs now in less than $1\%$. The results of increasing $\beta_{\text{shift}}$ to 52 are shown in table \ref{table:2}, implying that the problem without restrictions performs much better for the retailer as the profits reached are much higher than those of $\beta_{\text{shift}} = 0$ for the restricted problem.
\begin{table}[H]
\centering
\begin{tabular}{|l|l|l|}
\hline
                   & $\chi = 0$ & $\chi = 1$ \\ \hline
Free problem       & 530.34     & 514.01     \\ \hline
Restricted problem & 137.64     & 132.54     \\ \hline
\end{tabular}
\caption{Expected profits in pounds for the different problems and the different values of $\chi$}
\label{table:2}
\end{table}

\begin{figure}
\begin{subfigure}{\linewidth}
\centering
\includegraphics[width=\textwidth]{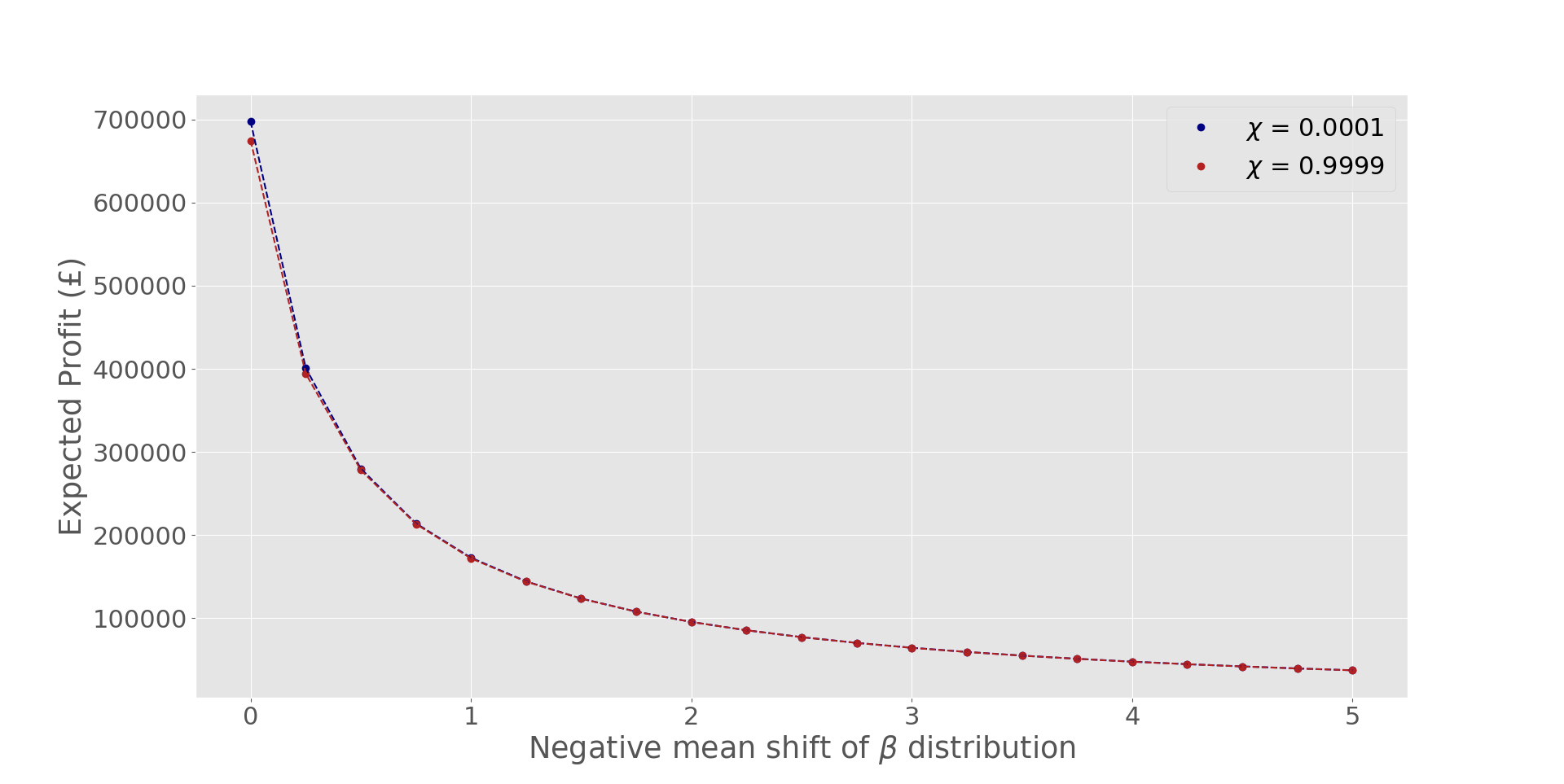}
\caption{}
\label{fig:sub1_1}
\end{subfigure}\\[1ex]
\begin{subfigure}{.5\linewidth}
\centering
\includegraphics[width=\textwidth]{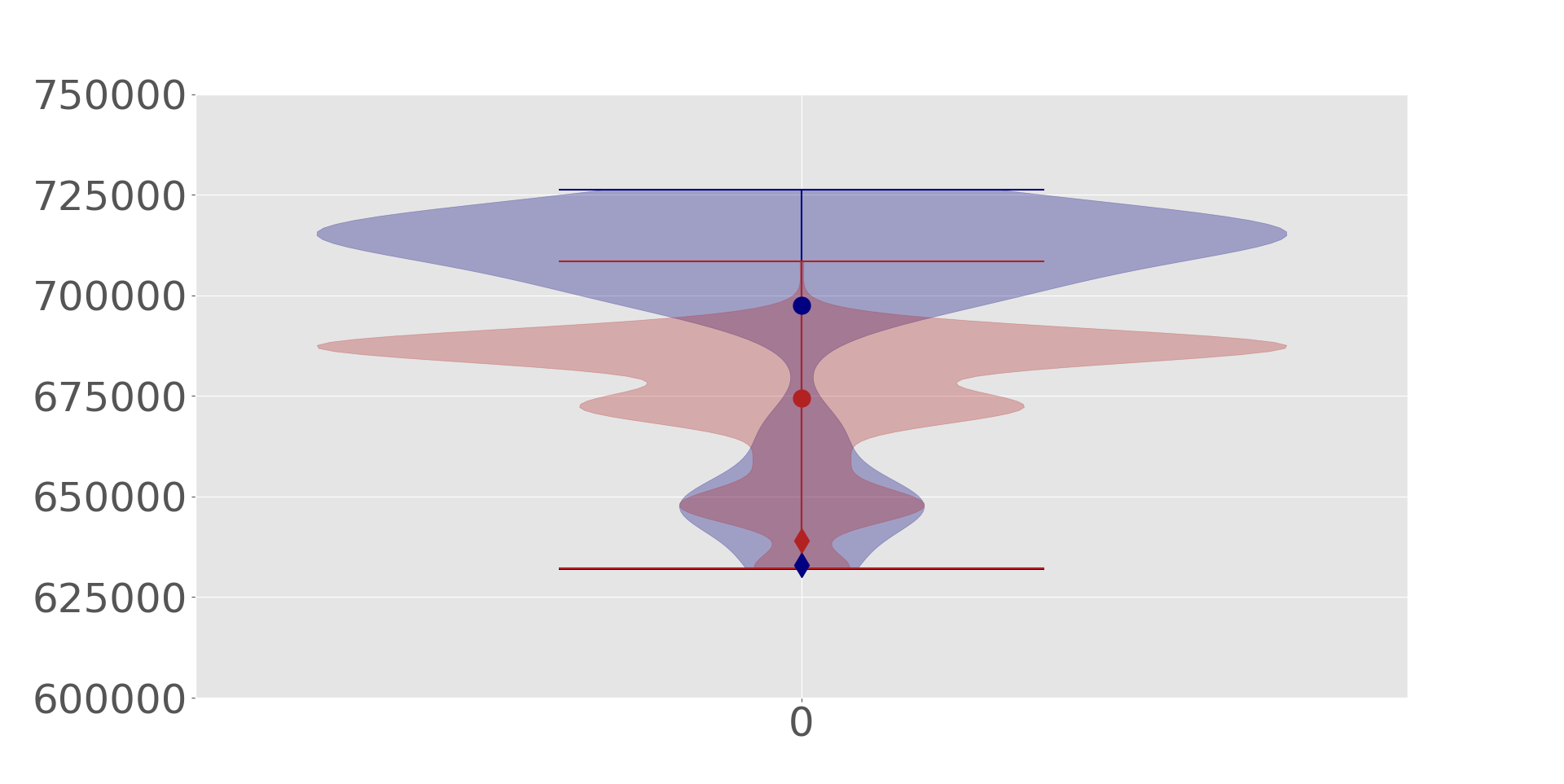}
\caption{}
\label{fig:sub2_2}
\end{subfigure}
\begin{subfigure}{.5\linewidth}
\centering
\includegraphics[width=\textwidth]{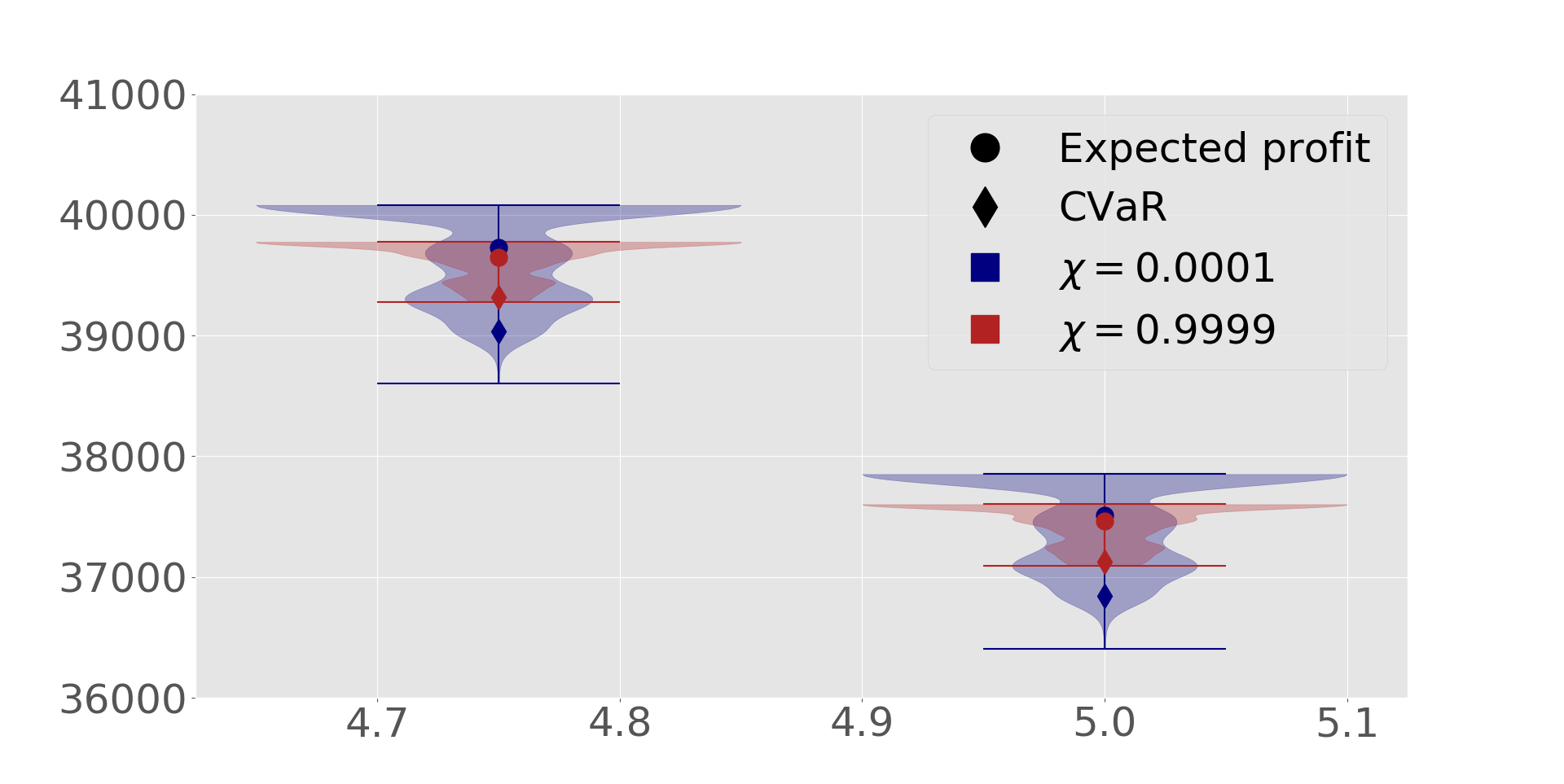}
\caption{}
\label{fig:sub3}
\end{subfigure}
\caption{Expected profit as a function of $\beta_{\text{shift}}$ (a). Violin plot for $\beta_{\text{shift}} = 0$ (subfigure (b)), and for $\beta_{\text{shift}} = 4.75$ and $\beta_{\text{shift}} = 5$ (subfigure (c)).}
\label{fig:profit_dist_free}
\end{figure}

\section{Closing remarks and future work}
\label{chap:conclusiones}

Although the dataset has been a limiting factor throughout this whole work because of being only slightly dependent on price, we have managed to formulate a restricted problem which could be adapted to these hypothetical low-sensitivity customers which produces both highly-competitive prices and safe profits for the retailer. We have also successfully included both stochasticity surging from data-driven scenario generation and the CVaR as a risk measure to increase the retailer's profit in the worst-case scenarios. Nonetheless, we believe that this low impact of the price variable stems only from the scarce options offered to retailers in the ToU tariff implemented in London in 2013 and that future datasets produced with more price options (or continuous prices) would not need these restrictions to formulate an economically viable problem. For this reason, we have artificially increased the sensitivity of customers to construct a simpler problem with fewer restrictions as the price can now be regulated only by means of the price sensitivity. This produces a much richer problem, both in expected profit for the retailer and in theoretical interest. 

The numerical results obtained from the case study indicate that risk hedging can be efficiently performed through the energy procurement strategy of the retailer. In this manner, the profit in worst scenarios can be increased at expenses of reducing slightly the total expected profit. However, the selling price offered to the clients is not influenced by the risk-aversion degree of the retailer in the analyzed case. 
It is interesting to note that the forward (base) contract is signed for a contract price smaller than that associated with the PPA contract in the risk-neutral case. In other words, the retailer is willing to pay a higher price for a PPA contract with uncertain production if this contract allows it to reduce the purchase of energy in pool during the central periods of the day, which are characterized by higher pool prices.
In this case, forward and PPA contracts are not signed for prices higher than 4.57 and 4.79 p/kWh, respectively. 
However, if risk is accounted for, the retailer purchases small amounts of energy from forward and PPA contracts even for high prices, up to 4.70 and 4.92 p/kWh for forward and PPA contracts, respectively.

It has been also observed that the price offered by the retailer is highly dependent on sensitivity of the costumers to prices. Assuming that the switch of electricity supplier by consumers is not considered in this model, a low price elasticity may lead to unreasonable results if the selling price offered by the retailer is not constrained. However, if the price elasticity of consumers is high enough, the retail price is reasonable even though this price is not bounded.

Future work is focused on considering simultaneously the sensitivity of consumers to prices and the possibility of switching supplier. Additionally, more complex price tariffs may be investigated to increase the profitability of the retailer under price-sensitive consumers.  

\section*{Acknowledgment}
The authors gratefully acknowledge the financial support from the Spanish government through projects PID2020-116694GB-I00 and PID2019-111211RBI00/
AEI/10.13039/501100011033, and from the Madrid Government (Comunidad de Madrid) under the Multiannual Agreement with UC3M in the line of ``Fostering Young Doctors Research'' (ZEROGASPAIN-CM-UC3M), and in the context of the V PRICIT (Regional Programme of Research and Technological Innovation.

\bibliographystyle{unsrt}      
\bibliography{Bibliography.bib}   

\begin{thebibliography}{10}

\bibitem{IRENA_2021}
IRENA.
\newblock \emph{Renewable Energy Capacity Statistics 2021}.
\newblock \url{https://www.irena.org/Statistics/Download-Data}, 2021.

\bibitem{ETS_2021}
European Commission.
\newblock \emph{EU Emissions Trading System}, 2021.
\newblock \url{https://ec.europa.eu/clima/policies/ets_en}.

\bibitem{Dong_2019}
S.~Dong, H.~Li, F.~Wallin, A.~Avelin, Q.~Zhang, and Z.~Yu.
\newblock Volatility of electricity price in denmark and sweden.
\newblock {\em Energy Procedia}, 2019:4331--4337, 2019.

\bibitem{GBmarket}
OFGEM.
\newblock \emph{The GB electricity wholesale market}.
\newblock
  \url{https://www.ofgem.gov.uk/electricity/wholesale-market/gb-electricity-wholesale-market}.

\bibitem{SmartMeter}
London Datastore.
\newblock \emph{SmartMeter Energy Consumption}, 2013.
\newblock
  \url{https://data.london.gov.uk/dataset/smartmeter-energy-use-data-in-london-households}.

\bibitem{AKHAVANHEJAZI201891}
Hossein Akhavan-Hejazi and Hamed Mohsenian-Rad.
\newblock Power systems big data analytics: An assessment of paradigm shift
  barriers and prospects.
\newblock {\em Energy Reports}, 4:91--100, 2018.

\bibitem{articleyang}
Jiajia Yang, Junhua Zhao, Fushuan Wen, and Z.Y. Dong.
\newblock A model of customizing electricity retail prices based on load
  profile clustering analysis.
\newblock {\em IEEE Transactions on Smart Grid}, PP:1--1, 04 2018.

\bibitem{6266720}
Peng Yang, Gongguo Tang, and Arye Nehorai.
\newblock A game-theoretic approach for optimal time-of-use electricity
  pricing.
\newblock {\em IEEE Transactions on Power Systems}, 28(2):884--892, 2013.

\bibitem{Feng_2020}
Cheng Feng, Yi~Wang, Kedi Zheng, and Qixin Chen.
\newblock Smart meter data-driven customizing price design for retailers.
\newblock {\em IEEE Transactions on Smart Grid}, 11(3):2043--2054, 2020.

\bibitem{Carrion_2007}
M.~Carri\'on, A.~J. Conejo, and J.~M. Arroyo.
\newblock Forward contracting and selling price determination for a retailer.
\newblock {\em IEEE Transactions on Power Systems}, 22(4):2105–2114, 2007.

\bibitem{Carrion_2008}
M.~Carri\'on, A.~J. Conejo, and J.~M. Arroyo.
\newblock A bilevel stochastic programming approach for retailer futures market
  trading.
\newblock {\em IEEE Transactions on Power Systems}, 24(3):1446–1456, 2007 9.

\bibitem{Hatami_2011}
A.~Hatami, H.~Seifi, and M.~K. Sheikh-El-Eslami.
\newblock A stochastic-based decision-making framework for an electricity
  retailer: Time-of-use pricing and electricity portfolio optimization.
\newblock {\em IEEE Transactions on Power Systems}, 26(4):1808--1816, 2011.

\bibitem{Kettunen_2010}
J.~Kettunen, A.~Salo, and D.~W. Bunn.
\newblock Optimization of electricity retailer's contract portfolio subject to
  risk preferences.
\newblock {\em IEEE Transactions on Power Systems}, 25(1):117--128, 2010.

\bibitem{Garcia-Bertrand_2013}
R.~Garc\'ia-Bertrand.
\newblock Sale prices setting tool for retailers.
\newblock {\em IEEE Transactions on Smart Grid}, 4(4):2028--2035, 2013.

\bibitem{Nojavan_2017}
S.~Nojavan, K.~Zare, and B.~Mohammadi-Ivatloo.
\newblock Optimal stochastic energy management of retailer based on selling
  price determination under smart grid environment in the presence of demand
  response program.
\newblock {\em Applied Energy}, 187:449--464, 2017.

\bibitem{Deng_2020}
T.~Deng, W.~Yan, S.~Nojavan, and K.~Jermsittiparsert.
\newblock Risk evaluation and retail electricity pricing using downside risk
  constraints method.
\newblock {\em Energy}, 192:116672, 2020.

\bibitem{Electricgrid}
Office of~electricity.
\newblock \emph{Demand Response}.
\newblock
  \url{https://www.energy.gov/oe/activities/technology-development/grid-modernization-and-smart-grid/demand-response}.

\bibitem{Hyndmanonline}
Rob~J. Hyndman and G.~Athanasopoulos.
\newblock \emph{Forecasting: principles and practice, 2nd edition}, 2018.
\newblock \url{OTexts.com/fpp2}.

\bibitem{Hyndman12}
Shu Fan and Rob~J. Hyndman.
\newblock Short-term load forecasting based on a semi-parametric additive
  model.
\newblock {\em IEEE Transactions on Power Systems}, 27(1):134--141, 2012.

\bibitem{libro}
Antonio~J. Conejo, Miguel Carrión, and Juan~M. Morales.
\newblock {\em Decision Making Under Uncertainty in Electricity Markets}.
\newblock Springer, 2010.

\bibitem{Rockafellar}
R~Rockafellar and Stan Uryasev.
\newblock Conditional value-at-risk for general loss distributions.
\newblock {\em Journal of Banking $\&$ Finance}, 26:1443--1471, 07 2002.

\bibitem{Rockafellar00optimizationof}
R.~Tyrrell Rockafellar and Stanislav Uryasev.
\newblock Optimization of conditional value-at-risk.
\newblock {\em Journal of Risk}, 2:21--41, 2000.

\bibitem{Hyndman10}
Rob Hyndman and Shu Fan.
\newblock Density forecasting for long-term peak electricity demand.
\newblock {\em Power Systems, IEEE Transactions on}, 25:1142 -- 1153, 06 2010.

\bibitem{Pineda2010ScenarioRF}
S.~Pineda and A.~Conejo.
\newblock Scenario reduction for risk-averse electricity trading.
\newblock {\em Iet Generation Transmission \& Distribution}, 4:694--705, 2010.

\bibitem{CLT}
Ping Ma Xinlian Zhang Xin Xing Jingyi Ma Michael~W. Mahoney.
\newblock Asymptotic analysis of sampling estimators for randomized numerical
  linear algebra algorithms.
\newblock {\em arxiv.org}, 2020.

\end{thebibliography}

\section*{Appendices}
\begin{appendix}
\section{Heuristic argument for the usage of the Hajek-Sidak CLT}
\label{Anexo:CLT}

\begin{theorem}[Hajek-Sidak Central Limit Theorem \cite{CLT}]
Let $X_1,\dots, X_n$ be independent and identically distributed (i.i.d.) random variables such that $\mathbb{E}[X_i] = \mu$ and $Var[X_i] = \sigma ^2$ are both finite.
Define $Tn = d_1X_1 + \dots + d_nX_n$, then:
\begin{equation}
  \frac{T_n-\mu \sum d_i}{\sigma \sqrt{\sum d_i^2}} \to \mathcal{N}(0,1)  
\end{equation}
whenever the Noether condition:
\begin{equation}
    \frac{\text{max}_{1\leq i\leq n}d_i^2}{\sum_{i=1}^n d_i^2} \to 0, \quad as \quad n \to \infty
\end{equation}
is satisfied.
\end{theorem}
This theorem implies that the normal approximation used is valid as long as the entries of the matrix $n^{-1}(\mathbf{X}^T\mathbf{X})$ can be bounded from above and below as that implies that $\text{max}_{1\leq i\leq n}d_i^2 = \text{max}_{1\leq i\leq n} [n^{-1}(\mathbf{X}^T\mathbf{X})^{-1}\mathbf{X}^T]^2_{ij}$ must also be bounded while the bound from below ensures that when $n \to \infty$, $\sum_{i=1}^n d_i^2 = \sum_{j=1}^n[n^{-1}(\mathbf{X}^T\mathbf{X})^{-1}\mathbf{X}^T]^2_{ij}$ diverges for being an infinite sum of a positive non-zero constant, implying that the Noether condition is satisfied. To bound the entries of the matrix $n^{-1}(\mathbf{X}^T\mathbf{X})$ its eigenvalues are used. We now prove the upper bound for the matrix entries given that an upper bound for the maximum eigenvalue of $n^{-1}(\mathbf{X}^T\mathbf{X})$ exists.\\

\begin{theorem}
Let $\lambda_{\text{max}}<\infty$ be the maximum eigenvalue of a given unitarily diagonalizable matrix $\mathbf{A} \in \mathbb{R}^{k \times k}$, then $|\mathbf{A}_{ij}| \leq \kappa $ where $\kappa$ is some undetermined constant.
\end{theorem}
\begin{proof} Since $\mathbf{A}$ is unitarily diagonalizable, $\mathbf{A} = \textbf{U}\boldsymbol{\Lambda}\textbf{U}^T$ and using the frobenius norm:
\begin{equation}
    \norm{\textbf{A}}{}_F = \sqrt{\sum_{i=1}^k \sum_{j=1}^k \left|\textbf{A}_{ij}\right|^2} = \norm{\textbf{U}\boldsymbol{\Lambda}\textbf{U}^T}{}_F = \norm{\boldsymbol{\Lambda}}{}_F = \sqrt{\sum_{i=1}^k\left|\lambda_i\right|^2} \leq \sqrt{k}\lambda_{\text{max}}<\infty
\end{equation}
\end{proof}

\begin{theorem}
Let $\lambda_{\text{min}}>0$ be the minimum eigenvalue of a given unitarily diagonalizable matrix $\mathbf{A}\in \mathbb{R}^{k \times k}$, then $|\mathbf{A}_{ij}| > 0$.
\end{theorem}
\begin{proof}
\begin{equation}
    \norm{\textbf{A}}{}_F = \sqrt{\sum_{i=1}^k \sum_{j=1}^k \left|\textbf{A}_{ij}\right|^2} = \norm{\textbf{U}\boldsymbol{\Lambda}\textbf{U}^T}{}_F = \norm{\boldsymbol{\Lambda}}{}_F = \sqrt{\sum_{i=1}^k\left|\lambda_i\right|^2} \geq \sqrt{k}\lambda_{\text{min}}>0
\end{equation}
\end{proof}
Since the matrix $n^{-1}(\mathbf{X}^T\mathbf{X})$ is a normal matrix, it is also unitarily diagonalizable, therefore its entries can be bounded if an upper and lower bound for its eigenvalues is found. Proofs for bounds of eigenvalues of real data matrices are hard to produce and thus this argument is purely heuristic. As can be seen in the following figures, as we increase $n$ in our data matrix by resampling the matrix's rows via bootstrap, both the maximum and minimum eigenvalue appear to converge to a given value.
\begin{figure}[H]
    \centering
    \includegraphics[width=\textwidth]{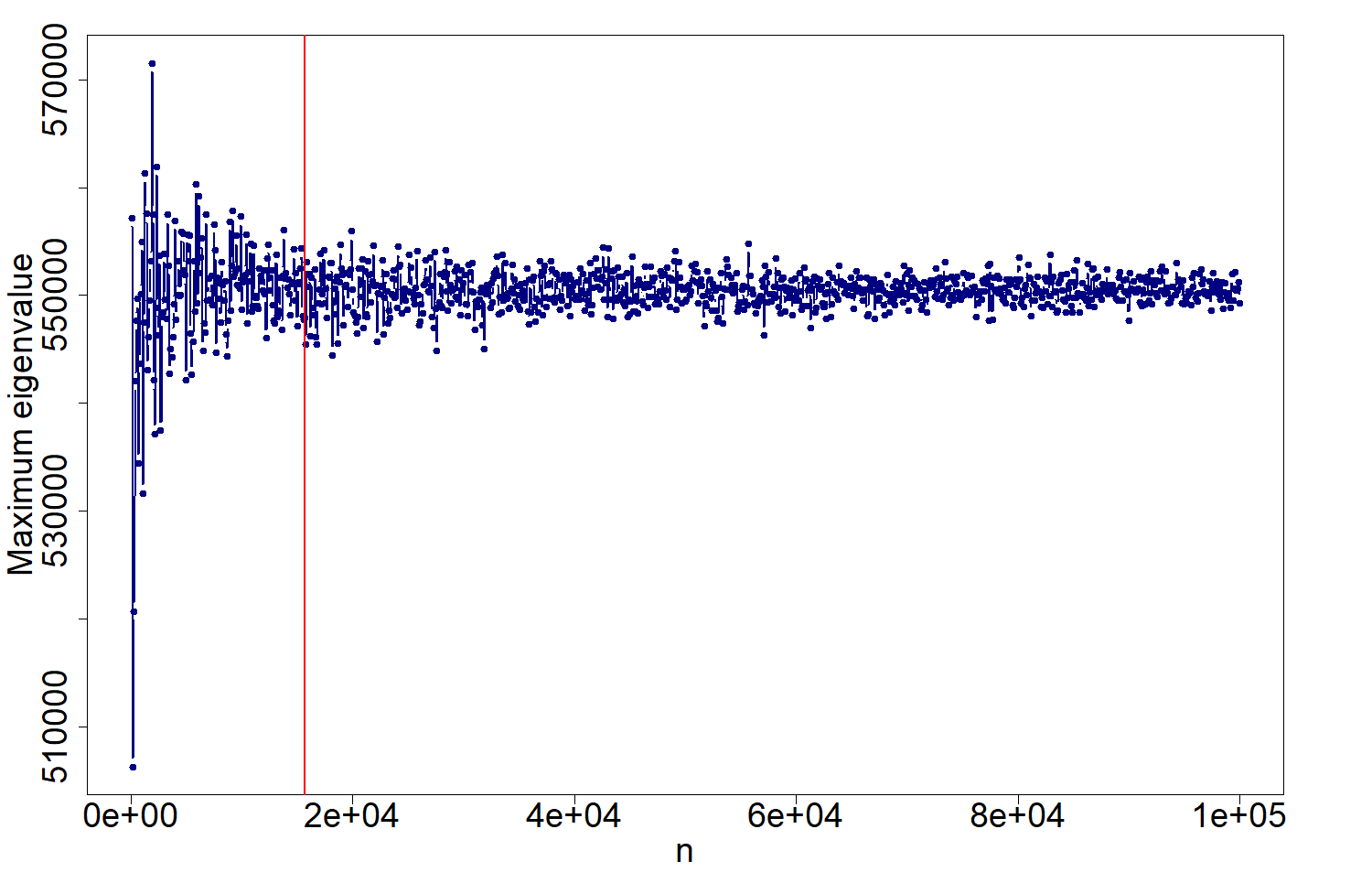}
    \caption{Plot of maximum eigenvalue as $n$ increases. The red line indicates the actual size of the matrix $\mathbf{X}$.}
\end{figure}
\begin{figure}[H]
    \centering
    \includegraphics[width=\textwidth]{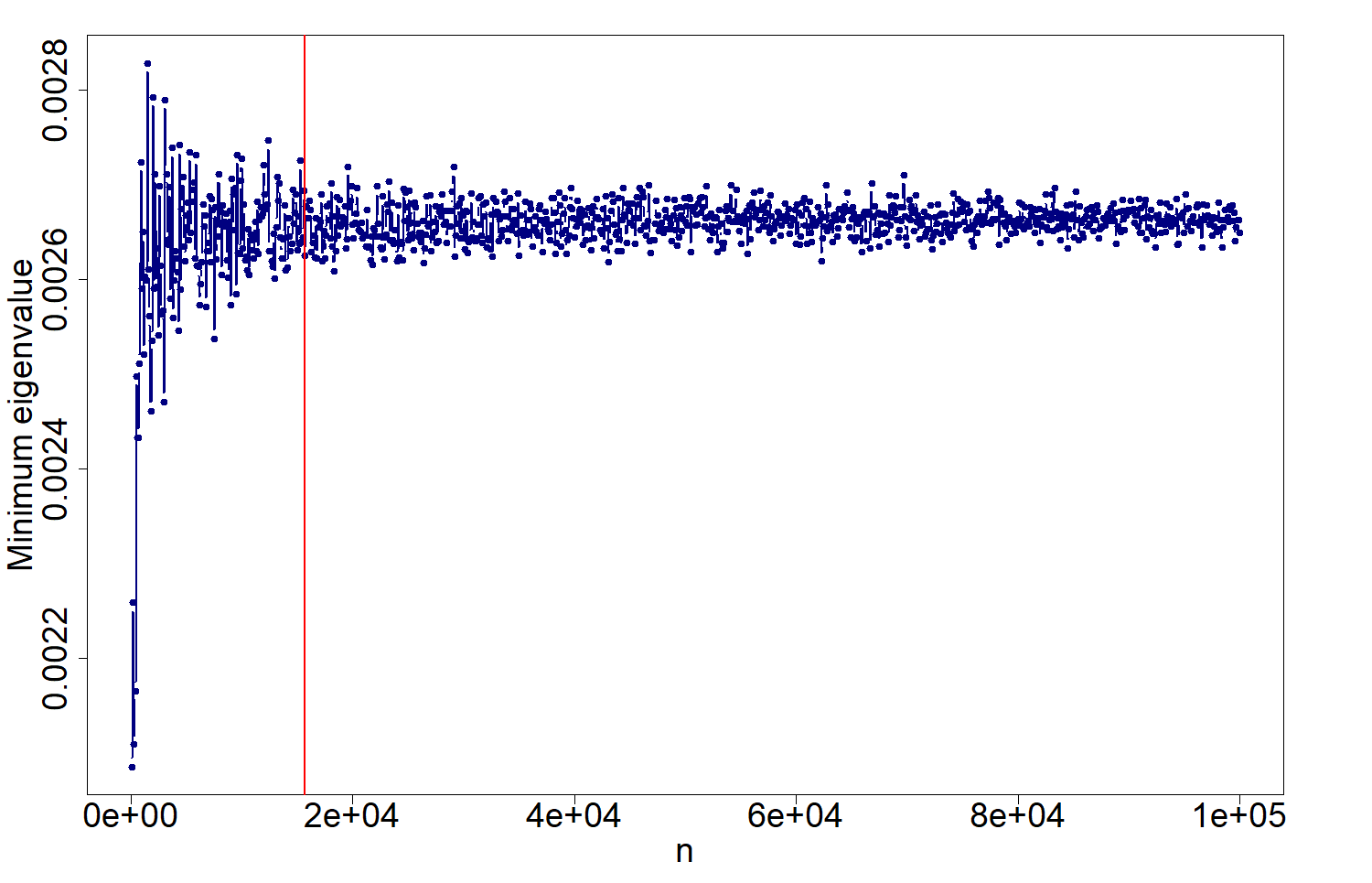}
    \caption{Plot of minimum eigenvalue as $n$ increases. The red line indicates the actual size of the matrix $\mathbf{X}$.}
\end{figure}
These results are considered sufficient for the application of the Hajek-Sidak CLT to the present data matrix.

\section{Scenario Generation for pool prices and solar availability}
\label{Anexo:Scenario}

To generate scenarios for both the pool prices and the solar availability we need a method which allows the generation of a large number of scenarios, as this set will be reduced later by other methods. Thus the following method is proposed. The appropriate data (either the pool price data spanning the years 2013-2020 or the solar availability data spanning the years 2010-2019) is split into a training, validation, test and comparison sets, the validation set being November 2013, the test set being December 2013 as in the rest of the present work, the training set being every date in the year 2013 excluding those in other sets and the comparison set being every date of every other year. For each date in the training set, the three closest dates from the comparison set are found by minimizing the euclidean norm between their respective scenarios. From these dates two separate empirical distributions of distance in days and distance in years are found. These distributions are exemplified in figure \ref{fig:test} for the solar availability data.

\begin{figure}[H]
\centering
\begin{subfigure}{.5\textwidth}
  \centering
  \includegraphics[width=\textwidth]{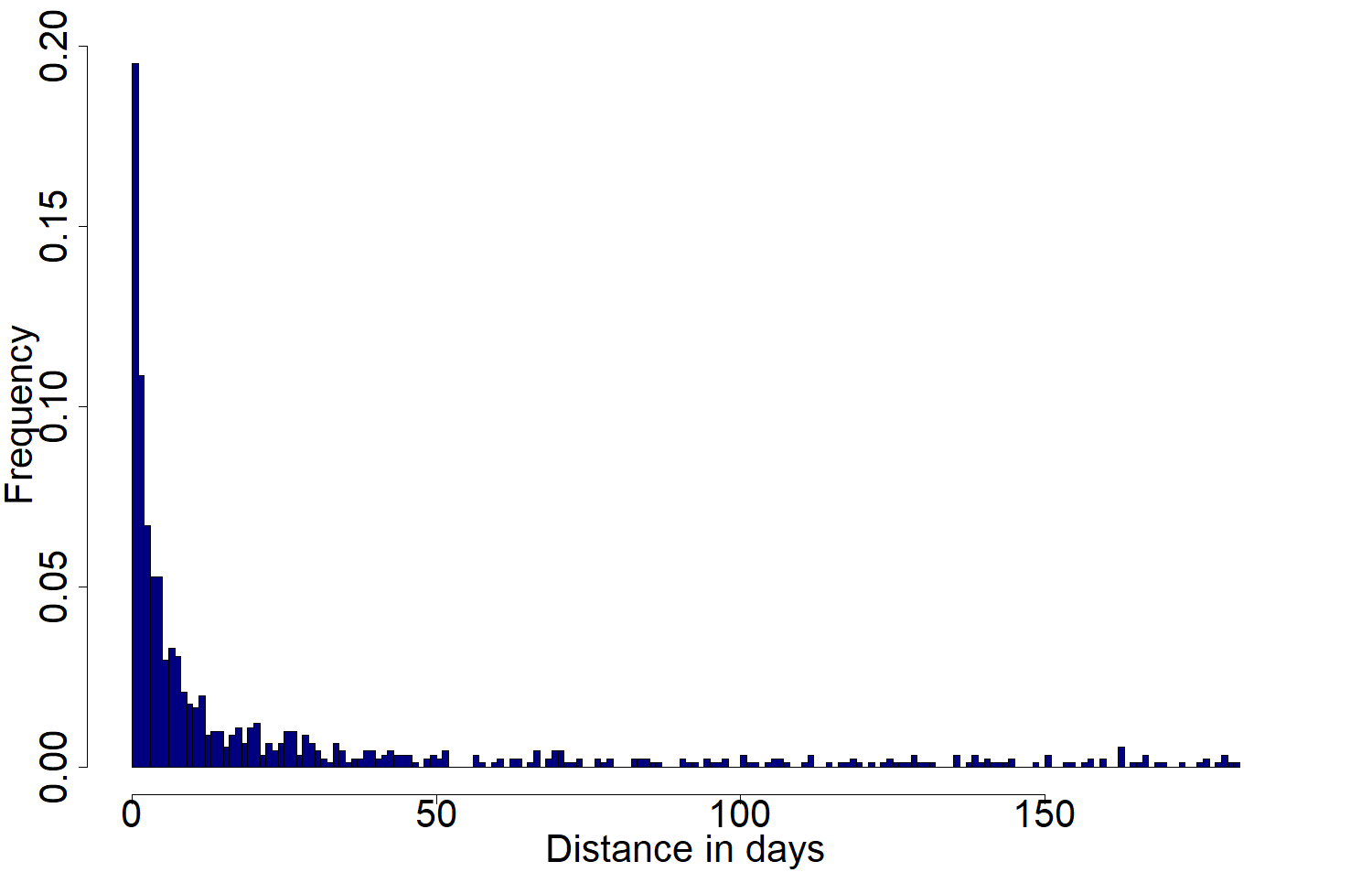}
  \caption{Distribution of distance in days}
  \label{fig:sub1}
\end{subfigure}%
\begin{subfigure}{.5\textwidth}
  \centering
  \includegraphics[width=\textwidth]{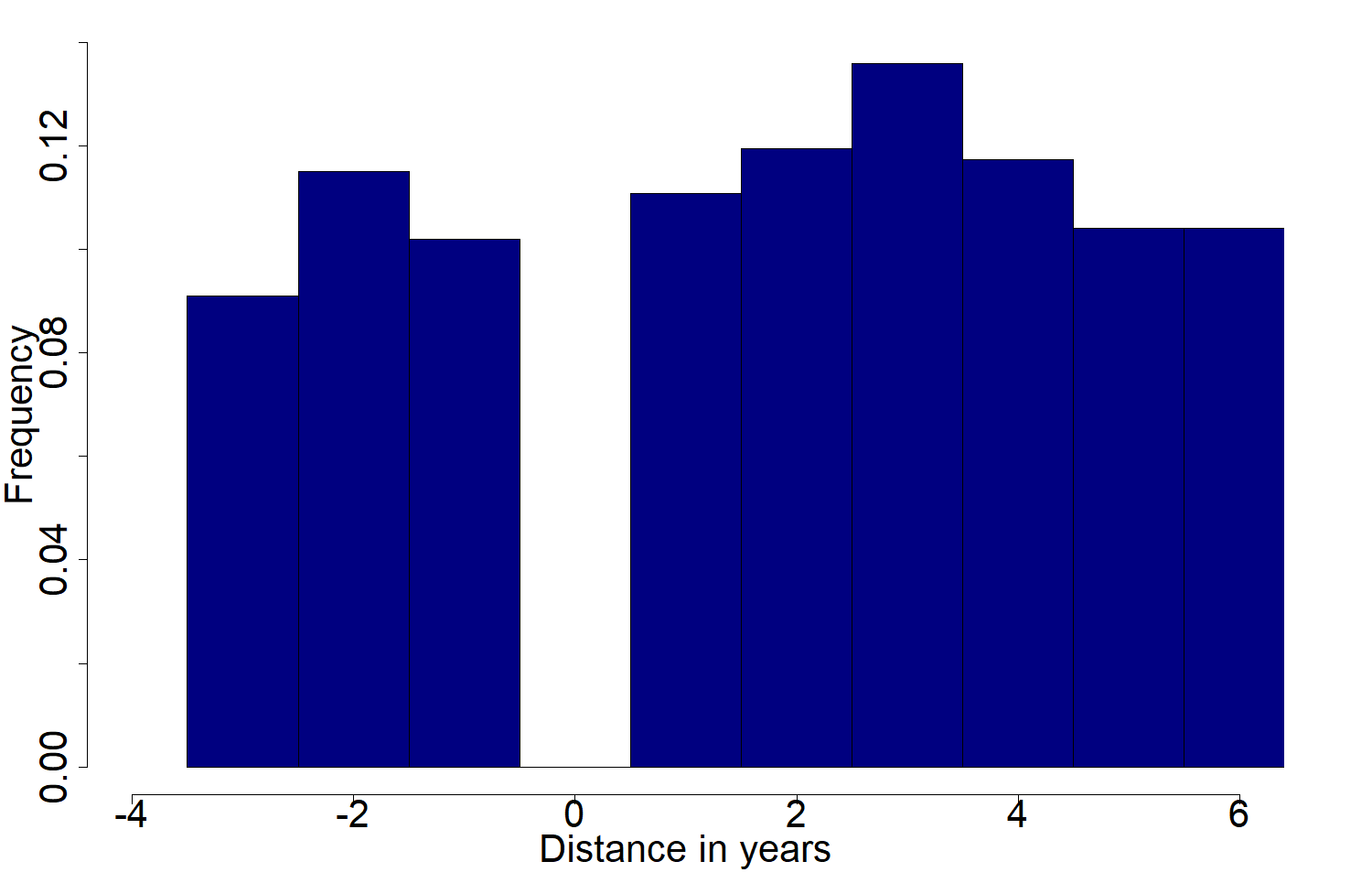}
  \caption{Distribution of distance in years}
  \label{fig:sub2}
\end{subfigure}
\caption{Distributions for the solar availability data. Since 2013 is excluded from the comparison set, a distance of 0 years is impossible.}
\label{fig:test}
\end{figure}

Finally, by sampling from these distributions, scenarios are generated for each date in the validation set by taking the corresponding scenario to the sampled date. In figures \ref{poolprices} and \ref{solaravailability} a comparison between this proposed method, completely random sampling and same date selection from different years is made for both the pool prices and solar availability data. Bars represent an average of euclidean distance between actual realizations for the validation set and the generated scenario.

\begin{figure}
    \centering
    \includegraphics[width=0.9\textwidth]{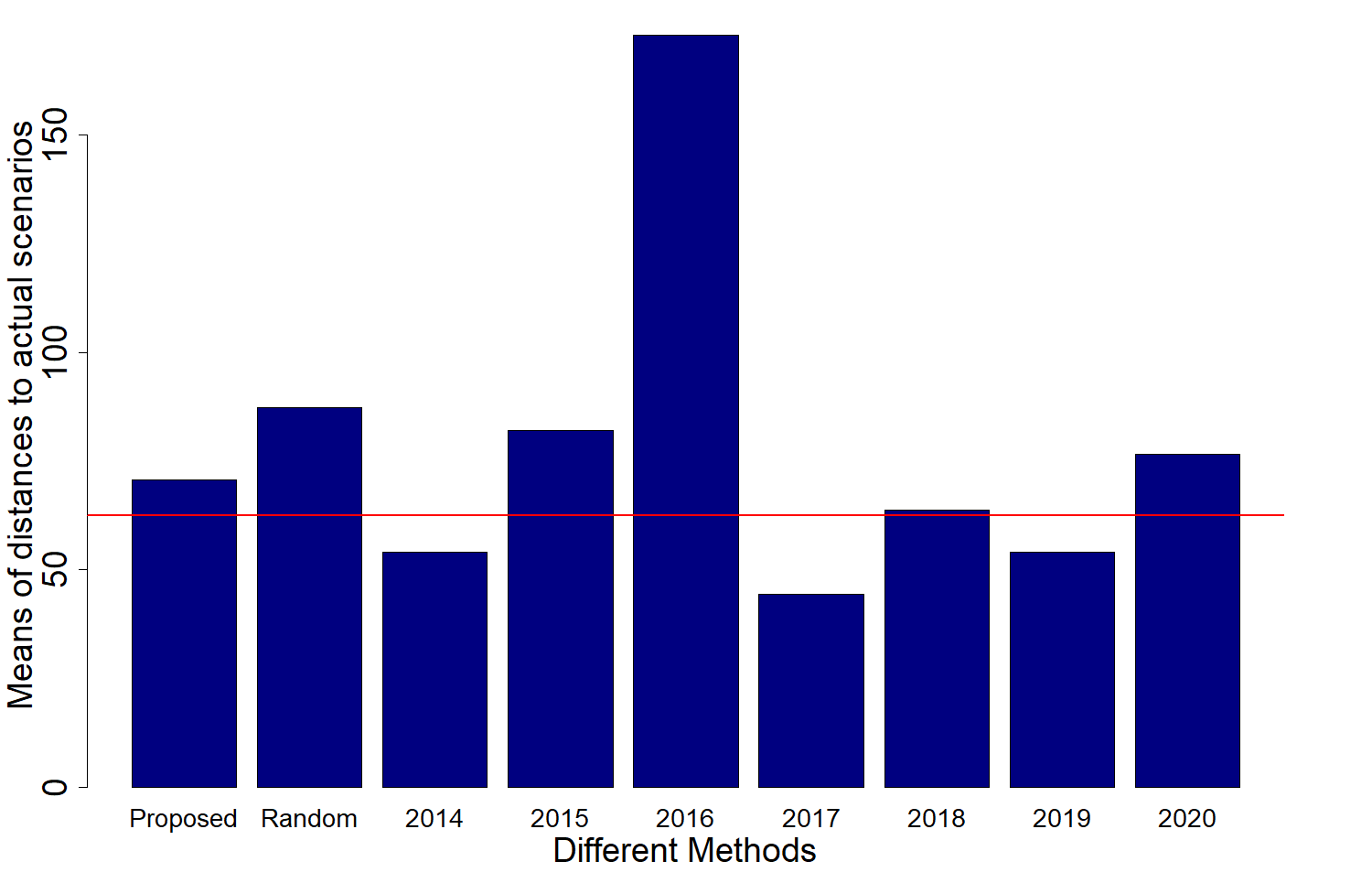}
    \caption{Barplot comparing the different methods for the pool prices data. The red line is the average of the years 2014 through 2020 except 2016 given its anomalous behaviour.}
    \label{poolprices}
\end{figure}
\begin{figure}
    \centering
    \includegraphics[width=0.9\textwidth]{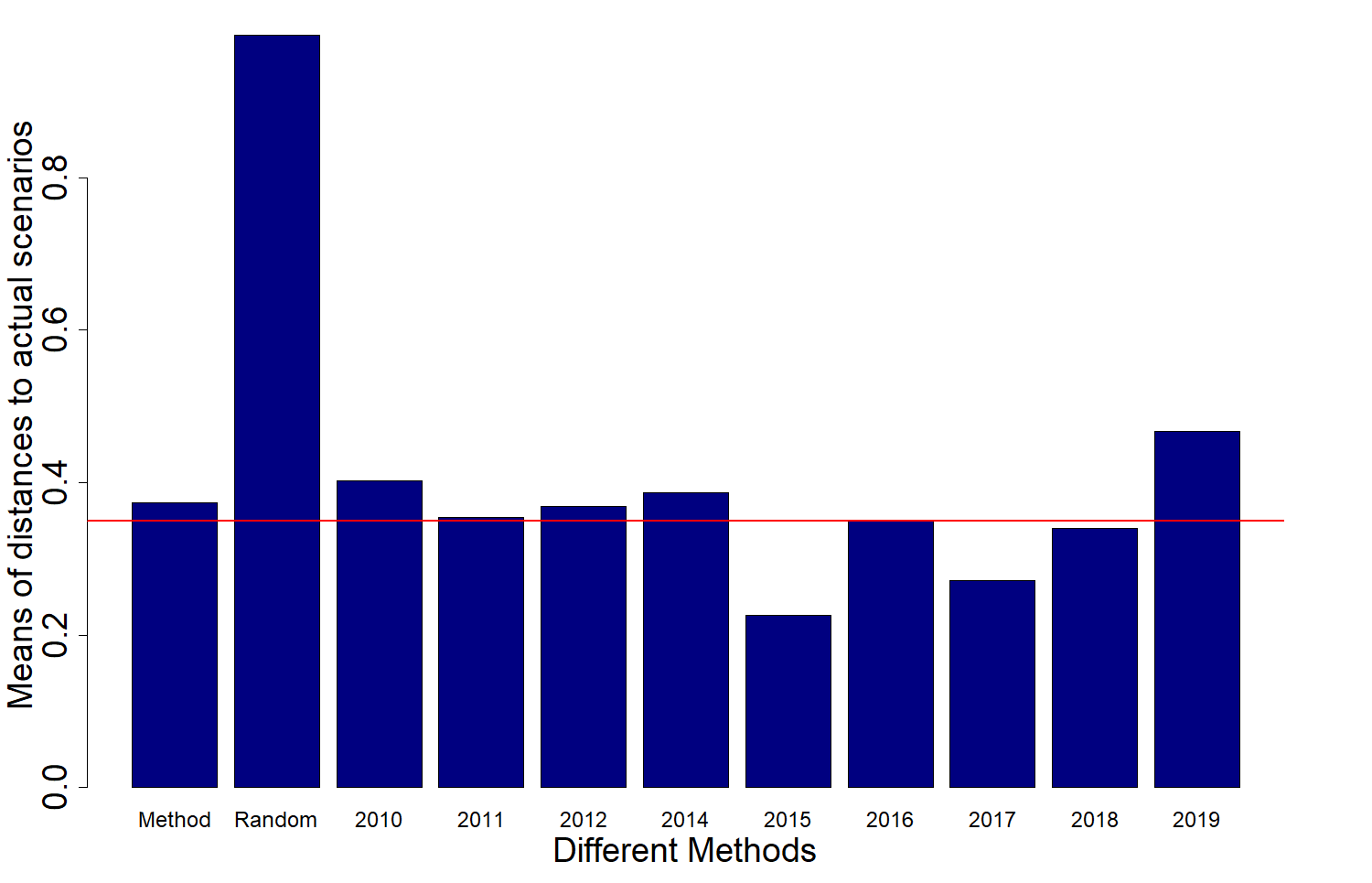}
    \caption{Barplot comparing the different methods for the solar availability data. The red line is the average of the years 2010 through 2019 except 2013 given that its the year used to compare to.}
    \label{solaravailability}
\end{figure}
As can be seen in these figures, while most years perform better than the proposed method, in particular their average, which could also be used to generate scenarios, this type of procedure is limited in the amount of scenarios it can produce (in this case it can only produce a maximum of 6 scenarios for each date for the pool price data and a maximum of 9 scenarios for the solar availability data). This, however, is not a problem for the proposed method which can handle far larger numbers of scenarios while performing noticeably better than the random method. This is why this approach is taken for the generation of pool price scenarios. Finally, the validation set is included in the training set and scenarios are generated for the test set.
\end{appendix}

\end{document}